\documentclass[a4paper, headings=small, bibliography=totoc, 12pt]{article}
\usepackage[utf8]{inputenc}
\usepackage[USenglish]{babel}
\usepackage[a4paper, head = 14.5pt, headsep = 10pt, footskip=40pt, twoside = false,%
            textheight = 25cm, textwidth = 15.2cm]{geometry}
\usepackage[pdftex,svgnames]{xcolor}
\usepackage[colorlinks, allcolors=gray]{hyperref}%=blue
\hypersetup{linktocpage=true, pdfpagemode=FullScreen}
\usepackage{amsmath, amssymb, amsthm, mathtools, esint, aligned-overset, mathabx}
\theoremstyle{plain}
    \newtheorem{thm}{Theorem}[section]
    \newtheorem{lem}[thm]{Lemma}

\theoremstyle{definition}
    \newtheorem{defi}[thm]{Definition}
    
    \newtheorem{ex}[thm]{Example}
\theoremstyle{remark}
    \newtheorem{rem}[thm]{Remark}
\newcommand{\N}{\mathbb{N}}
\newcommand{\R}{\mathbb{R}}

\newcommand{\V}{\mathbb{V}}
\newcommand{\Vh}{\mathbb{V}_h}

\newcommand{\Acal}{\mathcal{A}}
\newcommand{\Fcal}{\mathcal{F}}
\newcommand{\Scal}{\mathcal{S}}
\newcommand{\Tcal}{\mathcal{T}}
\newcommand{\Vcal}{\mathcal{V}}

\renewcommand{\div}{\mathrm{div}}   % divergence
\newcommand{\dx}{\mathrm{\,d}x}     % dx
\newcommand{\ds}{\mathrm{\,d}s}     % ds

\DeclareMathOperator{\diam}{diam}   % diam
\DeclareMathOperator{\interior}{int}% interior

\DeclareMathOperator{\essinf}{\mathrm{ess\,inf}}    % essential infimum
\DeclareMathOperator{\linh}{span}   % linear hull

\newcommand{\CR}{\mathrm{CR}}       % Crouzeix-Raviart
\newcommand{\eCR}{\textup{eCR}}     % enriched Crouzeix-Raviart
\newcommand{\mCR}{\textup{mCR}}     % modified enriched Crouzeix-Raviart
\newcommand{\RT}{\mathrm{RT}}       % Raviart-Thomas
\newcommand{\nc}{\mathrm{nc}}       % nonconforming
\newcommand{\pw}{\mathrm{pw}}       % piecewise
\newcommand{\es}{\mathrm{es}}       % extra-stabilzed
\newcommand{\GUB}{\mathrm{GUB}}
\newcommand{\GLB}{\mathrm{GLB}}
\renewcommand{\mid}{\mathrm{mid}}

\def\Xint#1{\mathchoice
    {\XXint\displaystyle\textstyle{#1}}%
    {\XXint\textstyle\scriptstyle{#1}}%
    {\XXint\scriptstyle\scriptscriptstyle{#1}}%
    {\XXint\scriptscriptstyle\scriptscriptstyle{#1}}%
\!\int}
\def\XXint#1#2#3{{\setbox0=\hbox{$#1{#2#3}{\int}$ }
    \vcenter{\hbox{$#2#3$ }}\kern-.6\wd0}}

\def\dashint{\Xint-}

\newcommand{\Vvert}[1]{{\left\vert\kern-0.25ex\left\vert\kern-0.25ex\left\vert #1 
    \right\vert\kern-0.25ex\right\vert\kern-0.25ex\right\vert}}

\usepackage{multirow}
\makeatletter
\def\hlinewd#1{%
\noalign{\ifnum0=`}\fi\hrule \@height #1 %
\futurelet\reserved@a\@xhline}
\makeatother

\usepackage{caption}
\usepackage{subcaption}
\usepackage{placeins}

\title{Old and new Schr\"odinger eigenvalue localisation}
\author{Carsten Carstensen\thanks{Department of Mathematics, 
Humboldt-Universit\"at zu Berlin, Unter
 den Linden 6, 10099 Berlin, Germany (\texttt{cc@math.hu-berlin.de}, \texttt{stiebert@math.hu-berlin.de})}
\quad and \quad Tim Stiebert\footnotemark[1]}
\date{\vspace{-14mm}}
\providecommand{\keywords}[1]{\noindent\textbf{Keywords.} #1}
\begin{document}

\maketitle

\begin{abstract}
    \noindent
    Unconditional guaranteed lower and upper eigenvalue bounds are mandatory for the understanding
    of the Schr\"odinger eigenvalue spectrum and its spectral gaps. While upper eigenvalue bounds are
    naturally induced by conforming discretisations, guaranteed lower eigenvalue bounds (GLB) are
    less immediate. This paper clarifies the adaptation of nonconforming GLB from the harmonic eigenvalue
    problem and discusses their comparison for general and piecewise constant potentials. A fine-tuned
    extra-stabilised scheme is proposed and found superior in numerical comparisons. This new direct
    calculation of GLB is compatible with adaptive mesh-refinement and successfully circumvents the
    appearance of maximal mesh-size parameters in former GLB based on post-processing. Computational
    benchmarks also investigate guaranteed upper eigenvalue bounds (GUB) for two-sided eigenvalue
    control by conforming test functions associated to the underlying nonconforming computations.
    A numerical comparison with GUB from additional lowest-order conforming finite element schemes
    shows competitive accuracy with less computational cost.
\end{abstract}

\vspace*{-0.2cm}

%%% 1. Introduction %%%
\section{Introduction} \label{sec:1}

%%% 1.1 Motivation %%%
\paragraph{Motivation.} \label{sec:1.1}

Hundred years of the Schr\"odinger eigenvalue problem (EVP) \cite{Sch26a} motivate a review
on guaranteed eigenvalue bounds for the lowest-order localisation of the spectrum. Given a
polyhedral bounded Lipschitz domain \(\Omega\subset\R^n\) in \(n \geq 2\) space dimensions
and some (trapping) potential \(V\), the Schr\"odinger eigenvalue problem seeks eigenpairs
\((\lambda,u)\) of \(-\Delta u + Vu = \lambda u\) in \(\Omega\) under homogeneous Dirichlet
boundary conditions. This is a fundamental model problem in quantum mechanics, an overview
over the spectrum is mandatory to determine the different energy levels of a quantum system.
More important in computational physics is that advanced (higher-order) schemes and estimators
require \emph{a-priori} knowledge about a spectral gap \cite{Gal14,Gal15,Gal17,CDMSV17,CDMSV18,%
CDMSV20, LV22}. While \emph{upper} eigenvalue bounds (GUB) follow immediately for conforming
finite element methods (FEM) from the Rayleigh-Ritz min-max principle \cite{BO91,Bof10}, the
computation of \emph{lower} eigenvalue bounds (GLB) is more involved, but indispensable for
two-sided control to identify spectral gaps.

\begin{table}[t]
	\hspace{-0.8cm}
	\begin{tabular}{l l c c c c}
		\hlinewd{1.5pt}
		method in EVP \eqref{eq:dSEVP} & \hspace{-2cm}GLB in Theorem
		& \hspace{-0.5cm}\(\Vh\) & \(P\) & \(Q\) & \(\gamma_h\) \\
		\hlinewd{1.5pt}
		\vspace*{-0.3cm} \\
		Crouzeix-Raviart (CR)           & \ref{thm:GLB_CR}  & \(\CR_0^1(\Tcal)\)  & \(\mathrm{id}\)
		& \(\mathrm{id}\) & \(\big(\sqrt{\varepsilon} + \sqrt{\delta/\lambda_\CR}\big)^2\lambda_\CR\)
		\medskip \\
		enriched Crouzeix-Raviart (eCR) & \ref{thm:GLB_eCR} & \(\eCR_0^1(\Tcal)\) & \(\mathrm{id}\)
		& \(\mathrm{id}\) & \(\displaystyle \delta' + \frac{\varepsilon'^2 \lambda_\eCR^2}{1 + \delta'
			+ \varepsilon' \lambda_\eCR}\) \medskip \\
		modified Crouzeix-Raviart (mCR) & \ref{thm:GLB_mCR} & \(\eCR_0^1(\Tcal)\) & \(\Pi_0\)
		& \(\mathrm{id}\) & \(\displaystyle \frac{\varepsilon' \varepsilon'' \lambda_\mCR^2}{1 +
			\varepsilon'' \lambda_\mCR}\) \medskip \\
		Raviart-Thomas (RT)             & \ref{thm:GLB_RT}  & \(\eCR_0^1(\Tcal)\) & \(\Pi_0\)
		& \(\Pi_0\) & \(\varepsilon' \, \lambda_\RT\) \\[5mm]
		extra-stabilised (sCR) in \eqref{eq:sCR} & \ref{thm:dGLB} & \eqref{eq:Ves}
		& \(\Pi_0\) & \(\mathrm{id}\) & \((\varepsilon''\lambda_\es-1)_+\) \smallskip \\
		\hlinewd{1.5pt}
	\end{tabular}
	\caption{Four Crouzeix-Raviart discretisations in \eqref{eq:dSEVP} and extra-stabilised scheme in
		\eqref{eq:sCR} with post-processing \(\gamma_h \) in \eqref{eq:glb} for parameters in \eqref{eq:params}.}
	\label{tab:1}
\end{table}

%%% Post-processed GLB %%%
\paragraph{Post-processed GLB.}
The first generation of \emph{guaranteed} lower eigenvalue bounds \cite{CG14} for the harmonic
eigenvalues relies on approximation properties of the nonconforming finite element interpolation \cite{%
CGR12}. It is highlighted in \cite{CZZ20,Liu15} that those guaranteed error bounds are
unconditional for higher eigenvalues: Given a discrete eigenvalue \(\lambda_h\) for an
exact eigenvalue \(\lambda\) of the same number \(k\in\N\) for a scheme and the post-%
processing \(\gamma_h\) from Table~\ref{tab:1}, we have
\begin{equation}
    \lambda_h/(1+ \gamma_h) \eqqcolon \GLB \leq \lambda.
    \label{eq:glb}
\end{equation}
It is marginal how big the maximal mesh-size \(h_{\max}\) of the underlying triangulation \(\Tcal
\) really is (although a finer mesh leads to a better bound), the assertion \(\GLB \leq \lambda\)
is guaranteed. This is a quantitative statement about the discretisation error and merely requires
exact solve of the algebraic eigenvalue problem, but is unconditional otherwise. The forthcoming
paper \cite{CS26+} extends the GLB for harmonic eigenvalues by conforming FEM \cite{LO13} to the
Schrödinger EVP \eqref{eq:SEVP}.

%%% Local mesh-refining %%%
\paragraph{Local mesh-refining.}
The Schr\"odinger eigenvalue problem highly depends on a prescribed potential \(V(x)\) and may exhibit 
localisation effects \cite{GN13} as illustrated for the Anderson localisation \cite{And58} in Figure
\ref{fig:groundStates}. The essence of this frequently observed phenomenon is that the eigenform of
interest is approximated very well by coarse meshes in parts of the domain, while a reasonable approximation 
enforces a fine mesh in other parts. This happens even for smooth or convex domains! The consequence is that
we encounter triangulations with highly different local mesh-sizes (e.g., measured in terms of the diameter
of the cells) and, moreover, we do \emph{not} want a uniform mesh. Then it is foreseen that \(h_{\max}\)
stays large and thereby ruins the efficiency of the valid GLB if local mesh-refining is mandatory: The
example of Subsection~\ref{sec:7.4}, for instance, provides a valid but useless \(\GLB \leq \lambda\).

%%% Direct GLB %%%
\paragraph{Direct GLB.}
The second generation of guaranteed error bounds \cite{CP23,CP24} circumvents this difficulty
by local mesh-parameters in extra-stabilisation terms. The fine-tuned parameters therein eventually lead
to new schemes that directly compute lower eigenvalue bounds! This has been explored for the harmonic
and biharmonic eigenvalues \cite{CP23,CP24} and is extended to the Schr\"odinger EVP in Section~\ref{sec:5}
of this paper.

%%% Contributions %%%
\paragraph{Contributions.}
The purpose of this  paper is threefold. First we adopt various GLB for the Schr\"odinger EVP and
thereby clarify the status quo on Schr\"odinger eigenvalue bounds in Section~\ref{sec:3}. Table~%
\ref{tab:1} displays four post-processings from nonconforming or mixed finite element calculations
that lead to GLB in \eqref{eq:glb} under various assumptions on the data. Many authors assume a
piecewise constant potential \(V \in P_0(\Tcal)\) \cite{arnold,AP19} and then particular bounds
simplify. We present a simultaneous proof in Section~\ref{sec:4} of the four post-processed GLB
of Section~\ref{sec:3}. Second, we establish the extra-stabilised EVP in Section~\ref{sec:5} for
it is compatible with local mesh-refining and leads to (empirical) optimal convergence rates in
benchmarks of Section~\ref{sec:7}. While we regard the mCR scheme from Table \ref{tab:1}
superior to sCR in the numerical examples in Section \ref{sec:7} on uniform meshes, the sCR scheme
appears more favourable for adaptive mesh-refinement for the Schr\"odinger eigenvalue localisation.
Third, we advertise and elaborate on the computation of GUB without direct use of conforming FEM via averaging \cite{CG14,HM14} in Section~\ref{sec:6}: Given \(k \in\N\) nonconforming eigenfunctions,
their conforming averages lead to a \(k \times k\) generalised algebraic eigenvalue problem with
\(k\)-th eigenvalue \(\mu_k \geq \lambda_k\) as GUB. An alternative method that allows for high
precision eigenvalue bounds is the Lehmann–Goerisch method, see \cite[Chapter~5]{Liu24} for harmonic
eigenvalues.

%%% Outline %%%
\paragraph{Outline.}
Section~\ref{sec:3} introduces the Schr\"odinger EVP and four nonconforming discretisations of
Table~\ref{tab:1} alongside necessary notation. Section~\ref{sec:3} corrects the recent GLB
\cite[Theorem 4.1]{Liu24} and compares it with \(\GLB_\CR\) presented in Section~\ref{sec:3}
beside other post-processed GLB from Table~\ref{tab:1}. Their comprehensive proofs follow in
in a unified frame in Section~\ref{sec:4}. The novel extra-stabilisation of enriched Crouzeix-Raviart
functions in Section~\ref{sec:5} allows for direct unconditional GLB even for piecewise constant
scalar diffusion. Section~\ref{sec:6} discusses simple and inexpensive GUB from averaging.
Computational benchmarks in Section~\ref{sec:7} compare the GLB of Table~\ref{tab:1} and GUB
from Section~\ref{sec:6} in praxi. Those comparisons provide striking empirical evidence for
the conjecture that the novel extra-stabilised EVP is the future low-order method of choice
for the Schr\"odinger EVP. The post-processed GUB from Section~\ref{sec:6} even match the accuracy
of direct GUB from conforming Courant EVP in all numerical experiments.

%%% 2. Prelimanieres and Notation %%%
\section{Preliminaries and Notation} \label{sec:2}

%%% Eigenvalue problem %%%
\subsection{Eigenvalue problem} \label{sec:1.2}
For the given potential \(V\in L^\infty(\Omega)\) we may and will assume without loss of
generality that \(\essinf_\Omega V = 0\) (otherwise shift the entire spectrum beforehand
by the essential infimum). The weak form of the Schr\"odinger eigenvalue problem seeks eigenpairs
\((\lambda,u)\in \R_+ \times H_0^1(\Omega)\) with \(L^2\)-normalized eigenstates \(\|u\|
\coloneqq \|u\|_{L^2(\Omega)} = 1\) such that
\begin{equation}
    a(u,v) + (u,v)_V = \lambda \, b(u,v) \qquad\textup{for all } v \in \V \coloneqq H_0^1(\Omega).
    \label{eq:SEVP}
\end{equation}
The energy scalar product \(a(\bullet,\bullet) \coloneqq (\nabla\bullet,\nabla\bullet)_\Omega\) and the
\(L^2\)-scalar product \(b(\bullet,\bullet) \coloneqq (\bullet,\bullet)_\Omega\) lead to Hilbert spaces
\((\V,a)\) and \((L^2(\Omega),b)\) in a Gelfand triple \(\V \hookrightarrow L^2(\Omega) \hookrightarrow
\V^\star\) with compact and dense embeddings. The semi-scalar product \((\bullet,\bullet)_V \coloneqq (
V\bullet,\bullet)_\Omega\) with induced semi-norm \(\|\bullet\|_V \coloneqq (\bullet,\bullet)_V^{1/2}\)
defines the norm \((\Vvert{\bullet}^2 + \|\bullet\|_V^2)^{1/2}\) that is equivalent to the energy norm
\(\Vvert{\bullet} \coloneqq a(\bullet,\bullet)^{1/2}\). The spectral theory for compact operators
\cite{kato} ensures countably many positive eigenvalues
\begin{equation}
    0 < \lambda_1 \leq \lambda_2 \leq \dots \leq \lambda_k \to \infty
    \qquad\textup{as } k \to \infty
    \label{eq:eVal}
\end{equation}
each of finite multiplicity. The eigenfunctions \((u_k)_{k\in\N}\) are \(a(\bullet,\bullet) + (\bullet,
\bullet)_V\) orthogonal and \(b\) orthonormal and the Rayleigh-Ritz min-max principle asserts that
the \(k\)-th eigenvalue \(\lambda_k\) and the set \(\Scal(k)\) of all \(k\)-dimensional subspaces of
\(\V\) satisfy
\begin{equation}
    \lambda_k = \min_{\V_k \in \Scal(k)} \max_{v \in \V_k \setminus \{0\}} \frac{a(v,v) + (v,v)_V}{b(v,v)}.
    \label{eq:minMax}
\end{equation}

%%% Discretisation %%%
\subsection{Discretisation} \label{sec:1.4}
Let \(\Tcal\) denote a regular triangulation of \(\Omega\) into non-degenerate
\(n\)-simplices (triangles in 2D). Four (out of five) competing discretisations
are written with a nonconforming test and ansatz space \(\Vh \subseteq \eCR_0^1
(\Tcal)\) as a subset of the piecewise quadratic enriched Crouzeix-Raviart
functions in \eqref{eq:eCR}. The resulting discrete problems seek algebraic
eigenpairs \((\lambda_h,u_h)\in \R_+ \times \Vh\) such that \(\|u_h\| = 1\) and
\begin{equation}
    a_\pw(u_h,v_h) + (Pu_h, Pv_h)_V = \lambda_h \, b(Qu_h, Qv_h)
    \qquad\textup{for all } v_h \in \Vh.
    \label{eq:dSEVP}
\end{equation}
The piecewise energy scalar product \(a_\pw(\bullet,\bullet) \coloneq \sum_{T
\in \Tcal} \int_T \nabla\bullet|_T \cdot \nabla\bullet|_T \dx\) is in fact a
scalar product on the finite dimensional spaces \(\Vh\) considered in this paper
\cite{BS08}. The linear operators \(P,Q\) in \eqref{eq:dSEVP} are either chosen
as the identity \(\mathrm{id}\) or as the \(L^2\) projection onto piecewise
constants \(P_0(\Tcal)\), which reads \(\Pi_0 v|_T \coloneqq \dashint_T v \dx
\coloneqq 1/\vert T\vert \int_T v \dx\). In all examples of this paper, the
coefficient matrix \(A \in \R^{N \times N}\) associated to the left-hand side
of \eqref{eq:dSEVP} is SPD for \(N = \dim\Vh\), while the symmetric matrix \(B
\in\R^{N \times N}\) for the right-hand side is SPD for \(Q = \mathrm{id}\). In
the remaining case \(Q = \Pi_0\) for \(\Vh = \eCR_0^1(\Tcal)\) the range of \(
Q\) is \(P_0(\Tcal)\) of dimension \(|\Tcal|= M < N = |\Tcal| + |\Fcal(\Omega)
|\) for the number of triangles \(|\Tcal|=M\) and the number of interior faces
\(|\Fcal(\Omega)|\). This results in \(M\) positive eigenvalues \(0 < \lambda_
{h(1)} \leq \dots \leq \lambda_{h(M)} < \infty\) and \(N-M = |\Fcal(\Omega)|\)
eigenvalues \(+\infty\) in the algebraic EVP \(Ax = \lambda Bx\). Table~\ref{tab:1}
provides an overview over the four nonconforming discretisations, their detailed
description follows in Section~\ref{sec:3}. The nonconforming interpolation
operators in Subsection~\ref{sec:4.1} give rise to a set of parameters in the
GLB of Table~\ref{tab:1}
\begin{equation}
    \begin{aligned}
    &\varepsilon    \coloneqq \|\kappa_\CR^2 h_\Tcal^2\|_\infty,
    &&\varepsilon'  \coloneqq \|C_P^2 h_\Tcal^2\|_\infty,
    &&\varepsilon'' \coloneqq \|\kappa_\eCR^2 h_\Tcal^2\|_\infty,
    &&\widetilde{\varepsilon} \coloneqq \|\kappa_h^2 h_\Tcal^2\|_\infty, \\
    &\delta \coloneqq \|\kappa_\CR^2 h_\Tcal^2 V\|_\infty,
    &&\delta' \coloneqq \|C_P^2 h_\Tcal^2 V\|_\infty,
    &&\widetilde{\delta} \coloneqq \|\kappa_h^2 h_\Tcal^2 V\|_\infty
    \end{aligned}
    \label{eq:params}
\end{equation}
with \(\kappa_h=\kappa_\CR\) if \(\Vh=\CR_0^1(\Tcal)\) and \(\kappa_h=\kappa_\eCR\)
if \(\Vh = \eCR_0^1(\Tcal)\) in \eqref{eq:dSEVP}. The Poincar\'e constant on convex
domains is \(C_P = 1/\pi\) \cite{PW60,Beb03} for \(n\geq 3\) in  \(\|f - \dashint_T
f \dx\|_{L^2(T)} \leq C_P h_\Tcal \Vvert{f}_T\) for \(f\in H^1(T)\) and \(T\in\Tcal
\). For a triangle \(T\), \(C_P = 1/j_{11}\) with the first positive root \(j_{11}=
3.8317059702\) of the Bessel function of the first kind \cite{LS10}, while \(C_P=1/
(\sqrt{2}\,\pi)\) for a right-isosceles triangle \cite{LS10}. For numerical bounds
on \(C_P\) on triangles of general shape see \cite{KL07}. The numerical examples in
Section~\ref{sec:7} with right-isosceles triangles run \(C_P = 1/(\sqrt{2}\pi)\).
Examples~\ref{ex:ICR} and~\ref{ex:IeCR} provide details on \(\kappa_\CR\) and \(\kappa_\eCR\).

%%% Notation %%%
\subsection{Notation}
Standard notation on Lebesgue and Sobolev spaces applies throughout this paper like
\(L^p(\Omega)\) with norm \(\|\bullet\|_p \coloneqq \|\bullet\|_{L^p(\Omega)}\) for
any \(1 \leq p \leq \infty\) and \((\bullet,\bullet)_\Omega \equiv (\bullet,\bullet
)_{L^2(\Omega)} \equiv \int_\Omega\bullet\bullet\dx\) is the \(L^2\) scalar product,
while \((\bullet,\bullet)_V \equiv(V\bullet,\bullet)_\Omega \equiv\int_\Omega\bullet
V\bullet\dx\) abbreviates the \(V\)-weighted semi-scalar product. The first-order
Sobolev space \(H^1(T)\) abbreviates \(H^1(\interior(T))\) for a compact simplex \(
T\) with non-void interior \(\interior(T)\). The vector space \(H^1(\Tcal)\coloneqq
\{v \in L^2(\Omega) : v|_T \in H^1(T) \textup{ for all } T \in\Tcal\}\) consists of
piecewise \(H^1\) functions with respect to some non-displayed triangulation \(\Tcal
\) of \(\Omega \subset \R^n\) into simplices. It is equipped with the semi-norm \(
\Vvert{\bullet}_\pw^2 \coloneqq a_\pw(\bullet,\bullet) = (\nabla_\pw\bullet,\nabla_
\pw\bullet)_\Omega\), where the piecewise gradient is also understood with respect
to \(\Tcal\). Let \(P_k(\Tcal) \coloneqq \{v \in L^2(\Omega) : v|_T \in P_k(T)
\textup{ for all } T \in \Tcal\}\) denote the set of piecewise polynomials of (total)
degree at most \(k \in\N_0\). The notation \(\vert\bullet\vert\) is context depending
and denotes either the Euclidean length of a vector, the cardinality of a finite set,
or the \(n\)-dimensional Lebesgue measure of a subset of \(\R^n\), but also the area
\(|F|\) of a face \(F\), or the length \(|E|\) of an edge \(E\). Let \(\Tcal\) denote
a shape-regular triangulation of \(\Omega\) into compact \(n\)-simplices \(T \in\Tcal
\) of positive diameter \(h_T = \diam(T) > 0\) with outer unit normal \(\nu_T\) on \(
\partial T\). Define the local mesh-size function \(h_\Tcal \in P_0(\Tcal)\) through
\(h_\Tcal|_T \coloneqq h_T\). For any \(T \in \Tcal\), let \(\Fcal(T)\) denote the set
of its \(n+1\) sides (edges in 2D) and \(\Vcal(T)\) the set of its \(n+1\) vertices.
Define the sets \(\Fcal \coloneqq \bigcup_{T\in\Tcal} \Fcal(T)\) respectively \(\Vcal
\coloneqq \bigcup_{T \in\Tcal} \Vcal(T)\) of all sides respectively all vertices. The
associated sets of all interior or boundary vertices (resp.\ interior or boundary sides)
are abbreviated by \(\Vcal(\Omega)\) or \(\Vcal(\partial\Omega)\) (resp.\ \(\Fcal(
\Omega)\) or \(\Fcal(\partial\Omega)\)). Assign each face \(F \in \Fcal\) with a unit
normal vector \(\nu_F\) of fixed orientation. On any interior face \(F \in \Fcal(T_+)
\cap \Fcal(T_-)\) shared by the simplices \(T_\pm\in\Tcal\), this induces a labelling
of \(T_\pm\) through \(\nu_F = \pm\nu_{T_\pm}|_F\). With this sign convention, the jump
\([v]_F\) of a piecewise Lipschitz continuous function \(v\) across \(F\) reads \([v]_F
(x) \coloneqq v|_{T_+}(x) - v|_{T_-}(x)\) at \(x \in F = \partial T_+ \cap \partial T_-
\). Owing to the homogeneous boundary conditions, the jump \([v]_F \coloneqq v(x)\) at
\(x \in F \in \Fcal(\partial\Omega)\) simplifies and \(\nu_F\) points outwards of the
domain \(\Omega\).

%%% Discretisation %%%
\section{Four discrete EVP and post-processed GLB} \label{sec:3}

This subsection displays the four methods in \eqref{eq:dSEVP} and the associated GLB
of Table~\ref{tab:1}, while Section~\ref{sec:4} gives the unified and comprehensive
proofs.

%%% CR %%%
\subsection{Crouzeix-Raviart} \label{sec:3.1}
The nonconforming Crouzeix-Raviart finite element \cite{CR73} consists of piecewise affine functions that
are continuous at all faces midpoints
\begin{equation}
    \CR_0^1(\Tcal) \coloneqq \big\{ v_\CR \in P_1(\Tcal) : [v_\CR]_F(\mid(F)) = 0 \textup{ for all } F
    \in \Fcal \big\}.
    \label{eq:CR}
\end{equation}
Recall the definition of the jumps on the boundary \(\partial\Omega\) reflects
homogeneous boundary conditions in the sense of Crouzeix-Raviart. The Crouzeix%
-Raviart EVP seeks \((\lambda_\CR,u_\CR) \in \R_+ \times \CR_0^1(\Tcal)\) with
\(\|u_\CR\| = 1\) such that \eqref{eq:dSEVP} holds with \(P =\mathrm{id}= Q\).
The discrete problem \eqref{eq:dSEVP} is a generalised algebraic EVP with SPD
coefficient matrices and leads to \(N = \dim \CR_0^1(\Tcal) = |\Fcal(\Omega)|
\) many positive eigenvalues \(0 < \lambda_{\CR(1)} \leq \dots \leq \lambda_{
\CR(N)}\), while we follow the convention \(\lambda_{\CR(N+1)} = \lambda_{\CR
(N+2)} = \dots = +\infty\). Recall the constants \(\varepsilon=\|\kappa_\CR^2 
h_\Tcal^2\|_\infty\) and \(\delta= \|\kappa_\CR^2 h_\Tcal^2 V\|_\infty\) from
\eqref{eq:params}.

\begin{thm}[GLB from CR] \label{thm:GLB_CR}
    The \(k\)-th exact eigenvalue \(\lambda =\lambda_k\) in \eqref{eq:SEVP} and
    the \(k\)-th Crouzeix-Raviart eigenvalue \(\lambda_\CR = \lambda_{\CR(k)}\)
    in \eqref{eq:dSEVP} with \(P = \mathrm{id} = Q\) and \(\Vh=\CR_0^1(\Tcal)\)
    for the same number \(k \in \N\)  satisfy
    \[
        \GLB_\CR \coloneqq \frac{\lambda_\CR}{1 + \big(\sqrt{\varepsilon} + \sqrt{\delta/\lambda_\CR} \big)^2
        \lambda_\CR} \leq \lambda.
    \]
\end{thm}

\noindent
The proof of Theorem~\ref{thm:GLB_CR} concludes in \eqref{eq:End_CR} in Subsection
\ref{sec:4.5} below with \(\gamma_h = \big(\sqrt{\varepsilon}+\sqrt{\delta/\lambda
_\CR} \big)^2 \lambda_\CR\) in Table~\ref{tab:1}. We compare with a recent result
in \cite[Theorem 4.1]{Liu24} and find its corrected version below is inferior to the
GLB of Theorem~\ref{thm:GLB_CR}.

\begin{thm}[GLB after \cite{Liu24}] \label{thm:GLB_Liu}
    The \(k\)-th exact eigenvalue \(\lambda= \lambda_k\) in \eqref{eq:SEVP} and
    the \(k\)-th Crouzeix-Raviart eigenvalue \(\lambda_\CR = \lambda_{\CR(k)}\)
    in \eqref{eq:dSEVP} with \(P = \mathrm{id} = Q\) and \(\Vh=\CR_0^1(\Tcal)\)
    for the same number \(k \in \N\) and the first Crouzeix-Raviart eigenvalue
    \(\mu =\lambda_{\CR(1)}\) in \eqref{eq:dSEVP} satisfy
    \[
        \GLB_\mu \coloneqq \frac{\lambda_\CR}{1 + \big(\sqrt{\varepsilon} + \sqrt{\delta/\mu} \big)^2
        \lambda_\CR} \leq \lambda.
    \]
\end{thm}
\begin{proof}
    Since \(\mu \equiv \lambda_{\CR(1)} \leq \lambda_{\CR(k)}\) for any \(k\in
    \N\) with equality for \(k= 1\), we have \(\GLB_\mu(1) = \GLB_\CR(1)\) and
    \(\GLB_\mu(k) \leq \GLB_\CR(k)\) for any \(k\geq 2\). Hence \(\GLB_\CR\leq
    \lambda\) in Theorem~\ref{thm:GLB_CR} concludes the proof. The
    book \cite{Liu24} presents a different approach that we revisit in Remark
   ~\ref{rem:interpolation} below. However, the proof of Lemma 4.2 in \cite{Liu24}
    displays a miscalculation: The first inequality on page 63, omits a factor
    \(\sqrt{c}\). The corrected lemma leads to \(\kappa_h = \varepsilon^{1/2}
    +\delta^{1/2}\mu^{-1}\) in Remark~\ref{rem:interpolation} (with different
    notation in \cite{Liu24}) and Theorem~\ref{thm:GLB_Liu} displays the correct
    version of \cite[Theorem 4.1]{Liu24}.
\end{proof}

%%% eCR %%%
\subsection{Enriched Crouzeix-Raviart} \label{sec:3.2}
An enrichment by quadratic bubble-functions allows an additional degree of freedom and leads to an
nonconforming interpolation operator with exact interpolation of the piecewise integral mean in the
enriched Crouzeix-Raviart space \cite{HHL14,HM15}; see also \cite{AB85,AC95,Che93} for the
equivalence of mixed and nonconforming FEM. Define for each simplex \(T\in\Tcal\) the nonconforming
quadratic bubble-function known from \cite{FS83} and the Marini identity \cite{Mar85,BC05} by
\begin{equation}
    \flat_T(x) \coloneqq
    \frac{n+2}{2} - \frac{n(n+1)^2(n+2)}{\sum_{P \neq Q \in \Vcal(T)} |P-Q|^2}
    \, |x-\mid(T)|^2
    \label{eq:bubble_def}
\end{equation}
at \(x \in T \in \Tcal\) and extend it by zero outside \(T\). Their linear hull
\[
    B(\Tcal) \coloneqq \linh\{v \in P_2(\Tcal) : v|_T \in \linh\{\flat_T\}
    \textup{ for all } T \in \Tcal\}
\]
defines the enriched Crouzeix-Raviart (eCR) space
\begin{equation}
    \eCR_0^1(\Tcal)\coloneqq \CR_0^1(\Tcal) \oplus B(\Tcal).
    \label{eq:eCR}
\end{equation}
The eCR EVP seeks \((\lambda_\eCR,u_\eCR)\in \R_+ \times \eCR_0^1(\Tcal)\) with
\(\|u_\eCR\| = 1\) such that \eqref{eq:dSEVP} holds with \(P =\mathrm{id}= Q\).
The discrete problem \eqref{eq:dSEVP} is a generalised algebraic EVP with SPD
coefficient matrices and hence has \(N = \dim\eCR_0^1(\Tcal) = |\Fcal(\Omega)|
+ |\Tcal|\) many positive eigenvalues \(0 < \lambda_{\eCR(1)} \leq \dots \leq
\lambda_{\eCR(N)}\). Recall the constants \(\varepsilon' = \|C_P^2 h_\Tcal^2\|_ \infty\) and \(\delta' = \|C_P^2 h_\Tcal^2 V\|_\infty\) from \eqref{eq:params}.

\begin{thm}[GLB from eCR] \label{thm:GLB_eCR_1}
    The \(k\)-th exact eigenvalue \(\lambda=\lambda_k\) in \eqref{eq:SEVP} and the
    \(k\)-th enriched Crouzeix-Raviart eigenvalue \(\lambda_\eCR = \lambda_{\eCR(k
    )}\) in \eqref{eq:dSEVP} with \(P =\mathrm{id}= Q\) and \(\Vh = \eCR_0^1(\Tcal )\), for the same number \(k\in\N\), and \(\zeta(s)\coloneqq 1 + \frac{\delta'
    }{s} - \delta' - s\), for \(0 < s < 1\), satisfy
    \begin{equation*}
        \GLB_{\eCR}^{(s)} \coloneqq \max_{0<s<1} \frac{\lambda_\eCR}{1 + \frac{\delta'}{s} +\frac{\varepsilon
        ^{\prime 2} \lambda_\eCR^2}{\zeta(s) + \varepsilon' \lambda_\eCR}} \leq \lambda.
    \end{equation*}
\end{thm}

\noindent
Equation \eqref{eq:End_eCR} concludes the proof of Theorem~\ref{thm:GLB_eCR_1} in
Subsection~\ref{sec:4.5} below. The optimal parameter \(s\) is the solution of a
one-dimensional minimization problem. For a piecewise constant potential \(V \in
P_0(\Tcal)\), however, the lower bound in Theorem~\ref{thm:GLB_eCR_1} simplifies
miraculously.

\begin{thm}[GLB from eCR for \(V \in P_0(\Tcal)\)] \label{thm:GLB_eCR}
    The \(k\)-th exact eigenvalue \(\lambda = \lambda_k\) in \eqref{eq:SEVP} and the
    \(k\)-th enriched Crouzeix-Raviart eigenvalue \(\lambda_\eCR = \lambda_{\eCR(k)}
    \) in \eqref{eq:dSEVP} with \(P = \mathrm{id} = Q\) and \(\Vh= \eCR_0^1(\Tcal)\) for the same number \(k \in \N\) satisfy for a piecewise constant potential \(V
    \in P_0(\Tcal)\) that
    \[
        \GLB_\eCR \coloneqq \frac{\lambda_\eCR}{1 + \delta' + \frac{\varepsilon^{\prime 2} \lambda_\eCR^2}{
        1 + \delta' + \varepsilon' \lambda_\eCR}} \leq \lambda.
    \]
\end{thm}

\noindent
Equation \eqref{eq:End_mCR} concludes the proof of Theorem~\ref{thm:GLB_eCR} in
Subsection~\ref{sec:4.5} below with \(\gamma_h=\delta'+\frac{\varepsilon^{\prime
2}\lambda_\eCR^2}{1 + \delta' + \varepsilon'\lambda_\eCR}\) in Table~\ref{tab:1}.

%%% mCR %%%
\subsection{Raviart-Thomas} \label{sec:3.3}
The Marini identity \cite[Theorem 3.3]{HM15} provides for a piecewise constant
potential \(V\in P_0(\Tcal)\) the equivalence of the Raviart-Thomas (RT) mixed
finite element method \cite{RT77} and the following nonconforming discretization:
Seek \((\lambda_\RT,u_\nc) \in \R _+ \times \eCR_0^1(\Tcal)\) with \(\|u_\nc\|
= 1\) such that \eqref{eq:dSEVP} holds with \(P = \Pi_0 = Q\) and \(\Vh = \eCR
_0^1(\Tcal)\). The discrete problem \eqref{eq:dSEVP} is a generalised algebraic
EVP with \(M = |\Tcal|\) positive eigenvalues \(0 < \lambda_{\RT(1)} \leq \dots
\leq \lambda_{\RT(M)} < \infty\) and \(|\Fcal(\Omega)|\) infinity eigenvalues.
The framework in \cite[Corollary 5.1]{Gal23} provides the following GLB for the
mixed RTEVP with \(\gamma_h = \varepsilon' \lambda_\RT\) in Table~\ref{tab:1}.

\begin{thm}[GLB from RT for \(V \in P_0(\Tcal)\)] \label{thm:GLB_RT}
    The \(k\)-th exact eigenvalue \(\lambda = \lambda_k\) in \eqref{eq:SEVP} an
    the \(k\)-th Raviart-Thomas eigenvalue \(\lambda_\RT= \lambda_{\RT(k)}\) in \eqref{eq:dSEVP} with \(P = \Pi_0 = Q\) and \(\Vh=\eCR_0^1(\Tcal)\) for the
    same number \(k \in \N\) satisfy for a piecewise constant potential \(V \in
    P_0(\Tcal)\) that
    \[
        \GLB_\RT \coloneqq \frac{\lambda_\RT}{1 + \varepsilon' \lambda_\RT} \leq \lambda.
    \]
\end{thm}

\begin{rem}[Marini identity] \label{rem:Marini}
    The right-hand sides \(f = (\lambda_\eCR-V) \Pi_0 u_\eCR \in P_0(\Tcal)\) and
    \(f=(\lambda_\RT-V)u_\RT \in P_0(\Tcal)\) lead in \cite[Theorem 3.3]{HM15} to
    the equivalence of the mixed RTEVP and the nonconforming eigenvalue problem in
    \eqref{eq:dSEVP} with \(P = \Pi_0 = Q\) and \(\Vh = \eCR_0^1(\Tcal)\).
\end{rem}

\begin{rem}[instability] \label{rem:instable}
	The linear operator \(\Pi_0 : \eCR_0^1(\Tcal) \to P_0(\Tcal)\) is surjective
	with \(\dim P_0(\Tcal) = |\Tcal|\). Consequently, \(\dim \ker\Pi_0 = |\Fcal(
    \Omega)|\). Undisplayed numerical experiment with the MATLAB routine \texttt{eigs}
    experience difficulty to solve the algebraic EVP \eqref{eq:dSEVP} with \(Q =
    \Pi_0\) for \(|\Tcal|\) large. RTEVP is implemented directly via the edge-oriented
    Raviart-Thomas basis \cite[Section 4]{BC05}.
\end{rem}

%%% mCR %%%
\subsection{Modified Crouzeix-Raviart} \label{sec:3.4}
The observation in Remark~\ref{rem:instable} motivates the modified (enriched) 
Crouzeix-Raviart (mCR) scheme in \eqref{eq:dSEVP}: Seek \((\lambda_\mCR,u_\mCR
) \in \R_+ \times \eCR_0^1 (\Tcal)\) with \(\|u_\mCR\| = 1\) such that \(P =
\Pi_0\) and \(Q = \mathrm{id}\) in \eqref{eq:dSEVP}. The algebraic eigenvalue
problem \eqref{eq:dSEVP} with SPD coefficient matrices has \(N = \dim\eCR_0^1
\) positive eigenvalues \(0 < \lambda_{\mCR(1)} \leq \dots \leq \lambda_{\mCR
(N)} <\infty\). Recall \(\varepsilon'' = \|\kappa_\eCR^2 h_\Tcal^2\|_\infty\)
from Table~\ref{tab:1}.

\begin{thm}[GLB from mCR for \(V \in P_0(\Tcal)\)] \label{thm:GLB_mCR}
    The \(k\)-th exact eigenvalue \(\lambda = \lambda_k\) in \eqref{eq:SEVP} and the
    \(k\)-th mCR eigenvalue \(\lambda_\mCR = \lambda_{\mCR(k)}\) in \eqref{eq:dSEVP}
    with \(P = \Pi_0\), \(Q = \mathrm{id}\), and \(\Vh = \eCR_0^1(\Tcal)\) for the 
    same number \(k \in \N\) satisfy for a piecewise constant potential \(V\in P_0
    (\Tcal)\) that
    \[
        \GLB_\mCR \coloneqq \frac{\lambda_\mCR}{1 + \frac{\varepsilon' \varepsilon'' \lambda_\mCR^2}{1 +
        \varepsilon'' \lambda_\mCR}} \leq \lambda.
    \]
\end{thm}

\noindent
Equation \eqref{eq:End_mCR} concludes the proof of Theorem~\ref{thm:GLB_mCR}
in Subsection~\ref{sec:4.5} below with \(\gamma_h=(\varepsilon'\varepsilon''
\lambda_\mCR^2)/(1 + \varepsilon''\lambda_\CR)\) in Table~\ref{tab:1}.
\begin{rem}[Comparison with CECR] \label{rem:CECR}
	The mCR EVP is equivalent to the composite enriched Crouzeix-Raviart (CECR) EVP
    in \cite[Subsection 4.1.3]{Liu24}. For a piecewise constant potential \(V\in
    P_0(\Tcal)\) a combination of \cite[Theorem 3.1]{Liu24} and \cite[Inequality
    (4.9)]{Liu24} provides that \(\GLB_\mathrm{CECR} \coloneqq \lambda_\mCR/(1 +
    \varepsilon''\lambda_\mCR)\leq\lambda\). (Here \(\lambda \coloneqq \lambda_k
    \) is the \(k\)-th eigenvalue in \eqref{eq:SEVP} and \(\lambda_\mCR\coloneqq
    \lambda_{\mCR(k)}\) is the \(k\)-th mCR eigenvalue in \eqref{eq:dSEVP}.)
    Elementary algebra reveals that the strict inequality \(\GLB_\mathrm{CECR} <
    \GLB_\mCR\) is equivalent to \(C_P^2-\kappa_\eCR^2<h_{\max}^{-2}\lambda_\mCR
    ^{-1}\). The experiments in Section \ref{sec:7} on right-isosceles triangles
    employ \(C_P^2-\kappa_\eCR^2=0.02846\) and we found the resulting inequality
    \(h_{\max}^2 \lambda_\mCR < 35.13703\) holds in all experiments displayed, 
    except for the coarsest mesh in Figure~\ref{fig:Square_20_harm}.
\end{rem}

%%% Guaranteed lower eigenvalue bounds %%%%%%%%%%%%%%%%%%%%%%%%%%%%%%%%%%%%%%%%%%%%%%%%%%%%%%%%%%%%%%%%%%%%
\section{Proof of guaranteed lower eigenvalue bounds} \label{sec:4}
\vspace*{-0.25cm} % for spacing in the preprint
This section simultaneously proves the GLB in Theorems~\ref{thm:GLB_CR},
\ref{thm:GLB_eCR_1},~\ref{thm:GLB_eCR}, and~\ref{thm:GLB_mCR}.

%%% Nonconforming interpolation %%%%%%%%%%%%%%%%%%%%%%%%%%%%%%%%%%%%%%%%%%%%%%%%%%%%%%%%%%%%%%%%%%%%%%%%%%%
\subsection{Nonconforming interpolation} \label{sec:4.1}
\vspace*{-0.1cm} % for spacing in the preprint

\begin{defi}[nonconforming interpolation] \label{def:I}
    The best-approximation operator \(I \in L(\V+\Vh;\Vh)\) with respect to
    the discrete energy scalar product \(a_\pw(\bullet,\bullet)\) is called
    nonconforming interpolation operator.
\end{defi}

\noindent
The following properties of \(I\) are fundamental in the subsequent analysis and
specify a positive constant \(\kappa_h > 0\) such that all \((v,v_h) \in \V \times
\Vh\) satisfy
\begin{itemize}
    \item[(I1)] \(a_\pw(v-Iv, v_h) = 0\)
    \item[(I2)] \(\Vert v-Iv \Vert_T \leq \kappa_h h_T \, \Vvert{v-Iv}_T\) for the
    size \(h_T = \diam(T)\) of the simplex \(T \in \Tcal\).
\end{itemize}
The enriched Crouzeix-Raviart interpolation operator \(I \coloneqq I_\eCR\) enjoys
the additional annihilation property
\begin{itemize}
    \item[(I3)] \(\Pi_0 Iv = \Pi_0 v\) for all \(v \in \V\).
\end{itemize}

\begin{rem}[interpolation in \(\V+\Vh\)] \label{rem:interpolation}
    Rather \(a_\pw(\bullet,\bullet)+(\bullet,\bullet)_V\) than \(a_\pw(\bullet,\bullet
    )\) is the natural energy scalar product in the Schr\"odinger eigenvalue problem
    \eqref{eq:SEVP}. This scalar product gives rise to a (different) best-approximation
    \(\mathcal{I} \in L(\V+\Vh;\Vh)\) characterized by \(\mathcal{I}v \in \Vh\) and
    \(a_\pw(v-\mathcal{I}v,w_h)+(v-\mathcal{I}v,w_h)_V=0\) for all \((v,w_ h)\in \V
    \times \Vh\). This framework leads in \cite{Liu15} to \(\GLB_\mu = \lambda_h/(1
    +\kappa_h^2 h_{\max}^2\lambda_h) \leq \lambda\) with \(\kappa_h\) from (I2) for
    \(\mathcal{I}\). However, the approach in \cite[Section 4.2.2]{Liu24} leads to
    GLB in Theorem~\ref{thm:GLB_Liu} that are less sharp than that of Theorem~\ref{thm:GLB_CR}.
\end{rem}

\begin{ex}[Crouzeix-Raviart] \label{ex:ICR}
    Let \(\lambda_P\) denote the barycentric coordinate associated to the vertex \(
    P \in \Vcal(T)\) opposite to the side \(F \in\Fcal(T)\) and \(\psi_F\coloneqq 1
    -n\lambda_P\) in \(T\in \Tcal\). Then \((\psi_F : F \in \Fcal(\Omega))\) is the
	side-oriented basis of \(\CR_0^1(\Tcal)\) and gives rise to the CR interpolation
    operator \(I_\CR : \V + \CR_0^1(\Tcal) \to \CR_0^1(\Tcal)\) defined by
    \[
        I_\CR v \coloneqq \sum_{F \in \Fcal(\Omega)} \Big(\dashint_F v \ds\Big)\psi_F
        \qquad\textup{for all } v \in \V + \CR_0^1(\Tcal).
    \]
    The Crouzeix-Raviart interpolation operator \(I_\CR\) satisfies (I1)--(I2) proven
    in \cite{CG14} with a constant \(\kappa_h = \kappa_\CR\) in (I2): \cite{Liu15}
    computes the bound \(\kappa_\CR\leq 0.1893\) for \(n = 2\) on arbitrary triangles,
    while \cite{CZZ20} provides
    \[
        \kappa_\CR^2 \coloneqq C_P^2 + \frac{1}{2n(n+1)(n+2)}
    \]
    for all \(n \geq 2\) with the Poincar\'e constant \(C_P \leq 1/\pi\). The
    numerical examples in Section~\ref{sec:7} run \(\kappa_\CR = 0.1893\).
\end{ex}

\begin{ex}[enriched Crouzeix-Raviart] \label{ex:IeCR}
    The quadratic bubble-functions \((\flat_T : T \in \Tcal)\) from \eqref{eq:bubble_def}
    satisfy by design \(\int_F \flat_T \ds =0\) and \(\dashint_T \flat_T \dx =1\) for all
    \(F \in\Fcal(T)\) and all \(T\in\Tcal\) \cite[Lemma 2.1]{HM14}. The eCR interpolation
    operator \(I_{\eCR} : \V + \eCR_0^1(\Tcal) \to \eCR_0^1(\Tcal)\) is defined by
    \[
        I_\eCR v \coloneqq I_\CR v + \sum_{T \in \Tcal} \Big(\dashint_T v - I_\CR v \dx \Big)\flat_T.
    \]
    The eCR interpolation \(I_\eCR\) satisfies (I1)--(I3) and in particular, for \(v \in
    \V\), that
    \begin{equation}
        \dashint_F I_\eCR v \ds = \dashint_F v \ds  \textup{ for all } F \in \Fcal
        \textup{ \ and \ }
        \dashint_T I_\eCR v \dx = \dashint_T v \dx  \textup{ for all } T \in \Tcal.
        \label{eq:IeCR}
    \end{equation}
    The orthogonality (I1) is proven in \cite[ Lemma 4.6.]{CG25}, while (I3) follows from
    \eqref{eq:IeCR}. A Poincar\'e inequality and (I3) reveal, for any \(T\in\Tcal\), that
    \[
        \|v-I_\eCR v\|_{L^2(T)} = \|(1-\Pi_0)(v-I_\eCR v)\|_{L^2(T)} \leq C_P h_T \, \Vvert{v-I_\eCR v}_T.
    \]
    Hence \(\kappa_\eCR \leq C_P\), while \(\kappa_\eCR \leq 0.1490\) is computed on arbitrary
    triangles for \(n = 2\) in \cite{XXL18}. The numerical examples in Section~\ref{sec:7} run
    \(\kappa_\eCR = 0.149\).
\end{ex}

%%% Setup %%%%%%%%%%%%%%%%%%%%%%%%%%%%%%%%%%%%%%%%%%%%%%%%%%%%%%%%%%%%%%%%%%%%%%%%%%%%%%%%%%%%%%%%%%%%%%%%
\subsection{Setup for four GLB} \label{sec:4.2}
Suppose that \(\lambda = \lambda_k\) is the \(k\)-th exact eigenvalue in \eqref{eq:SEVP} and
that \(\lambda_h=\lambda_{h(k)}\) is the \(k\)-th discrete eigenvalue in \eqref{eq:dSEVP} of
the same number \(k \in \N\). Let \((\lambda_1,u_1), \dots, (\lambda_k,u_k)\) denote the first
\(k\) eigenpairs of \eqref{eq:SEVP} with \(0 < \lambda_1 \leq\lambda_2 \leq\dots \leq\lambda_k
\equiv\lambda\). The orthonormal eigenfunctions \(u_1,\dots,u_k\) define the invariant subspace
\(E(k) \coloneqq \mathrm{span}\{u_1, \dots, u_k\}\) of dimension \(k\), while the nonconforming
space \(IE(k) \coloneqq \mathrm{span}\{Iu_1, \dots, Iu_k\}\) merely satisfies \(\dim IE(k) \leq
k\). In particular, if \(k \geq \dim\Vh + 1\), then \(IE(k)\) cannot be \(k\)-dimensional.
Consequently, we distinguish two cases.
\\
\\
\textbf{Case 1.} If \(\dim IE(k) \leq k-1\), then there exists some \(u \in E(k)\) such that
\(\|u\| = 1 \) and \(Iu = 0\). These properties of \(u\) and (I2) reveal that
\begin{equation}
    1
	=       \|u\|^2
	=       \|u-Iu\|^2
	\leq    \|h_\Tcal^2 \kappa_h^2\|_\infty \, \Vvert{u-Iu}_\pw^2
	=       \|h_\Tcal^2 \kappa_h^2\|_\infty \, \Vvert{u}_\pw^2.
	\label{eq:S11}
\end{equation}
Since \(u \in E(k)\) and \(\|u\| = 1\), there exist \(\xi_1, \dots, \xi_k \in \R\)
such that \(u = \sum_{j=1}^k \xi_j u_j\) and \(\xi_1^2 + \dots + \xi_k^2 = \|u\|^2
= 1\). This, the \(a(\bullet,\bullet) + (\bullet,\bullet)_V\) orthonormality of \(
u_1,\dots, u_k\), and \eqref{eq:eVal} establish
\begin{equation}
\Vvert{u}^2 + \|u\|_V^2 = \sum_{j=1}^k \xi_j^2 \, (\Vvert{u_j}^2 + \|u_j\|_V^2)
= \sum_{j=1}^k \xi_j^2 \, \lambda_j
\leq \max\{\lambda_1, \dots, \lambda_k\} \sum_{j=1}^k \xi_j^2 = \lambda_k.
\label{eq:S12}
\end{equation}
Recall \(\varepsilon\) and \(\varepsilon''\) from Table~\ref{tab:1} and define
\(\widetilde{\varepsilon} \coloneqq \varepsilon\) for \(\Vh = \CR_0^1(\Tcal)\)
respectively \(\widetilde{\varepsilon} \coloneqq \varepsilon''\) for \(\Vh =
\eCR_0^1(\Tcal)\). The combination of \eqref{eq:S11} and \eqref{eq:S12} reveals
\(1/\widetilde{\varepsilon} \leq \lambda\). Recall \(\GLB \equiv \lambda_h/(1 +
\gamma_h)\) with \(\gamma_h\) specified in Table~\ref{tab:1}. Elementary
calculations reveal in all examples for \(\gamma_h\) from Table~\ref{tab:1}
that the GLB from \eqref{eq:glb} is smaller or equal than \(1/\widetilde{
\varepsilon} \leq \lambda\).
\\
\\
\textbf{Case 2.} Suppose that \(\dim IE(k) = k \leq \dim \Vh\). The algebraic
min-max principle for the discrete problem \eqref{eq:dSEVP} implies
\begin{equation}
    \lambda_h
    \leq    \max_{v_h \in IE(k) \setminus \{0\}}
            \frac{a_\pw(v_h,v_h) + (Pv_h,Pv_h)_V}{\Vert v_h \Vert^2}.
    \label{eq:S21}
\end{equation}
The maximiser in \eqref{eq:S21} exists and can be written as \(Iu \in IE(k) \setminus
\{0\}\) for some \(u \in E(k)\) with \(\|u\| = 1\), whence \eqref{eq:S21} reads
\begin{equation}
    \lambda_h \, \|Iu\|^2 \leq \Vvert{Iu}_\pw^2 + \|PIu\|_V^2.
    \label{eq:S22}
\end{equation}
The Pythagoras identity \(\Vvert{Iu}_\pw^2 = \Vvert{u}^2 - \Vvert{u-Iu}_\pw^2\)
from (I2) and \(\Vvert{u}^2 \leq \lambda - \|u\|_V^2\) from \eqref{eq:S12} lead
in \eqref{eq:S22} to
\begin{equation}
    \lambda_h \, \|Iu\|^2 + \Vvert{u-Iu}_\pw^2 \leq \lambda + \|PIu\|_V^2 - \|u\|_V^2.
    \label{eq:S23}
\end{equation}

%%% Lower bound %%%%%%%%%%%%%%%%%%%%%%%%%%%%%%%%%%%%%%%%%%%%%%%%%%%%%%%%%%%%%%%%%%%%%%%%%%%%%%%%%%%%%%%%%%
\subsection{Lower bound on \(\|Iu\|^2\)} \label{sec:4.3}
This subsection establishes for any \(0 < t < 1\) the existence of some \(a \in
\{0,1\}\) in Table~\ref{tab:const} such that
\begin{equation}
    \big(1 - at - t(1-a)\varepsilon'\lambda\big) \, \lambda_h
    - (t^{-1}-1) \, \widetilde{\varepsilon} \, \Vvert{u-Iu}_\pw^2
    \leq    \Vert Iu \Vert^2.
    \label{eq:lower}
\end{equation}

\begin{lem}
    Suppose that \(I\) satisfies (I1)--(I2), then \(a = 0\) in \eqref{eq:lower}.
\end{lem}
\begin{proof}
    A reverse triangle inequality and \(\|u\| = 1\) establish the lower bound
    \[
        1 - 2\Vert u-Iu \Vert + \Vert u-Iu \Vert^2
        =       (1 - \Vert u-Iu \Vert)^2
        \leq    \Vert Iu \Vert^2.
    \]
    The weighted Young inequality \(2\|u-Iu\| \leq t + t^{-1} \|u-Iu\|^2\) holds
    for any \(0 < t < 1\) and leads in the last displayed inequality to \((1-t)-
    (t^{-1}-1) \|u-Iu\|^2 \leq \|Iu\|^2\). This, (I2), and \(1-t^{-1} < 0\) result
    in
    \[
        (1-t) - \widetilde{\varepsilon} \, (t^{-1}-1) \, \Vvert{u-Iu}_\pw^2
        \leq    (1-t) - (t^{-1}-1) \, \|u-Iu\|^2
        \leq    \|Iu\|^2.
        \qedhere
    \]
\end{proof}

\begin{lem}
    Suppose that \(I\) satisfies (I1)--(I3), then \(a = 1\) in \eqref{eq:lower}.
\end{lem}
\begin{proof}
    Elementary algebra, \(\|u\| = 1\), and \(\Pi_0 I u = \Pi_0 u\) from (I3)
    reveal
    \[
        1 +  \|u-Iu\|^2 - \|Iu\|^2 = 2 \, b(u,u-Iu) = 2 \, b((1-\Pi_0)u, u-Iu).
    \]
    The Cauchy inequality \(b((1-\Pi_0)u, u-Iu) \leq \|(1-\Pi_0)u\| \, \|u-Iu\|
    \), the Poincar\'e inequality \(\|(1-\Pi_0)u\| \leq \sqrt{\varepsilon'} \,
    \Vvert{u}\), and a weighted Young inequality lead to
    \begin{align*}
        2 &\, b((1-\Pi_0)u, u-Iu)
        \leq    2 \sqrt{\varepsilon'} \, \Vvert{u} \, \|u-Iu\| \\
        &\leq   t \varepsilon' \, \Vvert{u}^2 + t^{-1} \|u-Iu\|^2
        \leq    t \varepsilon' \lambda + \varepsilon'' \, t^{-1} \Vvert{u-Iu}_\pw^2
    \end{align*}
    with \(\Vvert{u}^2 \leq \lambda - \|u\|_V^2 \leq \lambda\) from \eqref{eq:S12}
    and (I3) in the last step. The combination of this with the first identity of
    this proof concludes the proof.
\end{proof}

%%% Upper bound %%%%%%%%%%%%%%%%%%%%%%%%%%%%%%%%%%%%%%%%%%%%%%%%%%%%%%%%%%%%%%%%%%%%%%%%%%%%%%%%%%%%%%%%
\subsection{Upper bound on \(\|PIu\|_V^2 - \|u\|_V^2\)} \label{sec:4.4}
This subsection establish three different upper bounds of the form
\begin{equation}
    \Vert PIu \Vert_V^2 - \Vert u \Vert_V^2 \leq b\lambda + c \Vvert{u-Iu}_\pw^2
    \label{eq:upper}
\end{equation}
for the constants \(b,c\) of Table~\ref{tab:const} for any parameter \(0 < s < 1\).

\begin{table}[t]
    \centering
    \begin{tabular}{c c c c c}
        \hlinewd{1.5pt}
        EVP \eqref{eq:dSEVP} & Theorem & \(a\) & \(b\) & \(c\) \\
        \hlinewd{1.5pt}
        \vspace*{-0.3cm} \\
        CR & \ref{thm:GLB_CR}       & \(1\) & \(\delta/s\) & \((1-s^{-1}) \, \delta + s\) \smallskip \\
        eCR & \ref{thm:GLB_eCR_1}   & \(0\) & \(\delta'/s\) & \((1-s^{-1}) \, \delta' + s\) \smallskip \\
        eCR & \ref{thm:GLB_eCR}     & \(0\) & \(\delta'\) & \(-\delta'\) \smallskip \\
        mCR & \ref{thm:GLB_mCR}     & \(0\) & \(0\) & \(0\) \\
        \hlinewd{1.5pt}
    \end{tabular}
    \caption{Parameters \(a,b,c\) in \eqref{eq:lower}--\eqref{eq:upper}.}
    \label{tab:const}
\end{table}

\begin{lem}[\(P = \mathrm{id}\)] \label{lem:4.7}
    Suppose (I1)--(I2) and \(P = \mathrm{id}\), then \eqref{eq:upper} holds with \(\widetilde{
    \delta}\coloneqq\|h_\Tcal^2 \kappa_h^2 V\|_\infty\) in \(b = \widetilde{\delta}/s\) and \(
    c = (1 - 1/s)\widetilde{\delta} + s\) for any \(0 < s < 1\).
\end{lem}
\begin{proof}
	Notice \(\|u-Iu\|_V^2 \leq \sum_{T\in\Tcal} \|V\|_{L^\infty(T)} \|u-Iu\|_{L^2(T)}^2\)
    so that (I2) reveals \(\|u-Iu\|_V^2 \leq \widetilde{\delta}\Vvert{u-Iu}_\pw^2\). This
    applies twice and elementary algebra and a Cauchy inequality result in
    \begin{align*}
        \|Iu\|_V^2 - \|u\|_V^2
        &=      \|u-Iu\|_V^2 + 2 \, (u, Iu-u)_V  \\
        &\leq   \widetilde{\delta} \, \Vvert{u-Iu}_\pw^2
                + 2 \, \|u\|_V \, \|u-Iu\|_V \\
        &\leq   \widetilde{\delta} \, \Vvert{u-Iu}_\pw^2
                + 2 \, \widetilde{\delta}^{1/2} \, \|u\|_V \, \Vvert{u-Iu}_\pw.
    \end{align*}
    A weighted Young inequality with \(0 < s < 1\) and \(\|u\|_V^2 \leq \lambda -\Vvert{u}^2\)
    from \eqref{eq:S12} establish
    \begin{align*}
        2 \, \widetilde{\delta}^{1/2} \, \|u\|_V \, \Vvert{u-Iu}_\pw
        &\leq   \frac{\widetilde{\delta}}{s} \, \|u\|_V^2 + s \Vvert{u-Iu}_\pw^2 \\
        &\leq   \frac{\widetilde{\delta}}{s} \lambda - \frac{\widetilde{\delta}}{s}
                \, \Vvert{u}^2 + s \Vvert{u-Iu}_\pw^2 \\
        &=      \frac{\widetilde{\delta}}{s} \lambda + \Big(s - \frac{\widetilde{\delta}}{s}\Big)
                \Vvert{u-Iu}_\pw^2 - \frac{\widetilde{\delta}}{s} \Vvert{Iu}_\pw^2
    \end{align*}
    with the Pythagoras identity \(\Vvert{u}^2 = \Vvert{Iu}_\pw^2 + \Vvert{u-Iu}_\pw^2\)
    from (I1) in the last step. The combination of the two displayed inequalities and \(
    -\widetilde{\delta} \Vvert{Iu}_\pw^2/s < 0\) concludes the proof of \eqref{eq:upper}
    for \(b = \widetilde{\delta}/s\) and \(c = (1-1/s)\widetilde{\delta}\). Recall \(
    \delta' = \|C_P^2 h_\Tcal^2 V\|_\infty\) from \eqref{eq:params}.
\end{proof}

\begin{lem}[\(P = \mathrm{id}\), (I3), \(V \in P_0(\Tcal)\)] \label{lem:4.8}
    Suppose (I1)--(I3), \(P = \mathrm{id}\), and a piecewise constant potential \(
    V \in P_0(\Tcal)\). Then \eqref{eq:upper} holds with \(b = \delta' = -c\).
\end{lem}
\begin{proof}
    The Pythagoras identity \(\|\Pi_0 Iu\|_{L^2(T)}^2 + \|(1-\Pi_0)Iu\|_{L^2(T)}^2
    =  \|Iu\|_{L^2(T)}^2\) for \(T \in \Tcal\), \(\Pi_0 Iu = \Pi_0 u\) from (I3),
    and \(\|\Pi_0u\|_{L^2(T)}^2 = \|u\|_{L^2(T)}^2 - \|u-\Pi_0u\|_{L^2(T)}^2 \leq
    \|u\|_{L^2(T)}^2\) imply
    \[
        \|Iu\|_{L^2(T)}^2
        =       \|\Pi_0 Iu\|_{L^2(T)}^2 + \|(1-\Pi_0)Iu\|_{L^2(T)}^2
        \leq    \|u\|_{L^2(T)}^2 + C_P^2 h_T^2 \, \Vvert{Iu}_T^2
    \]
    with a Poincar\'e inequality in the last step. Since \(V \in P_0(\Tcal)\), this
    leads to the global estimate \(\|Iu\|_V^2 - \|u\|_V^2 \leq \delta' \,\Vvert{Iu}
    _\pw^2\). The orthogonality (I1) and \eqref{eq:S12} establish \(\Vvert{Iu}_\pw^
    2 = \Vvert{u}^2 - \Vvert{u-Iu}_\pw^2 \leq \lambda - \Vvert{u-Iu}_\pw^2 - \|u\|_
    V^2\), whence
    \[
        \Vert Iu \Vert_V^2 - \Vert u \Vert_V^2
        \leq    \delta' \, \Vvert{Iu}_\pw^2
        \leq    \delta' \lambda - \delta' \Vvert{u-Iu}_\pw^2 - \delta' \Vert u \Vert_V^2
        \leq    \delta' \lambda - \delta' \Vvert{u-Iu}_\pw^2.
        \qedhere
    \]
\end{proof}

\begin{lem}[\(P = \Pi_0\), (I3), \(V \in P_0(\Tcal)\)] \label{lem:4.9}
    Suppose (I1)--(I3), \(P = \Pi_0\), and a piecewise constant potential
    \(V \in P_0(\Tcal)\). Then \eqref{eq:upper} holds with \(b = 0 = c\).
\end{lem}
\begin{proof}
    Since \(\Pi_0Iu = \Pi_0u\) by (I3) and \(\|\Pi_0u\|_{L^2(T)}^2 = \|u\|_
    {L^2(T)}^2 - \|u-\Pi_0 u\|_{L^2(T)}^2 \leq~\|u\|_{L^2(T)}^2\), it holds
    that
    \begin{equation}
        \|\Pi_0 Iu\|_V^2
        =       \|\Pi_0 u\|_V^2
        =       \sum_{T \in \Tcal} V|_T \|\Pi_0u\|_{L^2(T)}^2
        \leq    \sum_{T \in \Tcal} V|_T \|u\|_{L^2(T)}^2
        =       \|u\|_V^2,
        \label{eq:key}
    \end{equation}
    whence \(\|\Pi_0 Iu\|_V^2 - \|u\|_V^2 \leq 0\). This concludes the proof
    of \eqref{eq:upper} for \(b = 0 = c\).
\end{proof}

%%% Finish %%%%%%%%%%%%%%%%%%%%%%%%%%%%%%%%%%%%%%%%%%%%%%%%%%%%%%%%%%%%%%%%%%%%%%%%%%%%%%%%%%%%%%%%%%%%%
\subsection{Finish of the proof of the GLB} \label{sec:4.5}
The simultaneous proof of the GLB combines the lower bound \eqref{eq:lower}, the
upper bound \eqref{eq:upper}, and  \eqref{eq:S23} for the key inequality
\[
    \big(1 - at - t(1-a)\varepsilon'\lambda\big) \, \lambda_h + \big(1 - c -
    (t^{-1}-1) \, \widetilde{\varepsilon} \lambda_h  \big)  \, \Vvert{u-Iu}_\pw^2
    \leq    (1 + b) \, \lambda
    \tag{K}
    \label{eq:KI}
\]
for any \(0 < t < 1\). The choice
\[
    0 < t \coloneqq \frac{\widetilde{\varepsilon}\lambda_h}{1 - c + \widetilde{\varepsilon}\lambda_h} < 1
\]
leads to a vanishing prefactor of \(\Vvert{u-Iu}_\pw^2\) in \eqref{eq:KI} and so reveals
\begin{equation}
    \GLB = \frac{(1-at) \, \lambda_h}{1 + b + t (1-a) \varepsilon' \lambda_h} \leq \lambda.
    \label{eq:End_mCR}
\end{equation}
The coefficients in Table \ref{tab:const} conclude the proof of Theorems \ref{thm:GLB_eCR}
and \ref{thm:GLB_mCR}. Lemma \ref{lem:4.7} holds for any \(0 < s < 1\). The constants in
Table \ref{tab:const} lead for the CR EVP in \eqref{eq:End_mCR} to
\[
    \GLB_\CR
    =   \frac{\lambda_\CR}{\big(1 + \frac{\delta}{s}\big) \, \big(1 + \frac{\varepsilon\lambda_\CR}{
        1 + \frac{\delta}{s} - \delta - s}\big)}
    =   \frac{\lambda_\CR}{1 + \frac{\delta}{s} + \frac{\varepsilon \lambda_\CR}{1-s}}
    \leq \lambda.
\]
The lower bound attains its maximum at \(s \coloneqq \delta/(\delta + \sqrt{\delta\varepsilon
\lambda_\CR})\) and reads
\begin{equation}
    \GLB_\CR = \frac{\lambda_\CR}{1 + \delta + \varepsilon\lambda_\CR +
    2\sqrt{\varepsilon\delta\lambda_\CR}} \leq \lambda.
    \label{eq:End_CR}
\end{equation}
This concludes the proof of Theorem \ref{thm:GLB_CR}. The constants in Table \ref{tab:const}
reveal for the eCR EVP in \eqref{eq:End_mCR} that 
\begin{equation}
    \GLB_\eCR^{(s)} = \max_{0 < s < 1} \frac{\lambda_\eCR}{1 + \frac{\delta'}{s} + \frac{(\varepsilon')^2
    \lambda_\eCR^2}{1 + \frac{\delta'}{s} - \delta' - s + \varepsilon'\lambda_\eCR}} \leq \lambda.
    \label{eq:End_eCR}
\end{equation}
This concludes the proof of Theorem \ref{thm:GLB_eCR_1}. \qed

\begin{rem}[convergence of \(\gamma_h\)] \label{rem:Conv4pp}
    Note that every post-processing \(\gamma_h\) in Table \ref{tab:1} satisfies
    \(\gamma_h = \mathcal{O}(h_{\max}^2)\) independent of the regularity of the
    continuous (resp.\ discrete) eigenfunction \(u\) in \eqref{eq:SEVP}. Consequently,
    any GLB in Table \ref{tab:1} convergences with at least the same rate to \(\lambda
    \) as the respective discrete eigenvalue \(\lambda_h\) for uniform mesh-refinement.
\end{rem}

%%% Extra-stabilisation %%%%%%%%%%%%%%%%%%%%%%%%%%%%%%%%%%%%%%%%%%%%%%%%%%%%%%%%%%%%%%%%%%%%%%%%%%%%%%%%%%
\section{Direct GLB via extra-stabilisation} \label{sec:5}
An extra-stabilisation of Crouzeix-Raviart and Morley functions leads to direct GLB for
the (bi-)Laplacian in \cite{CP23,CP24}. This paper presents a first extra-stabilisation
of enriched Crouzeix-Raviart functions for \(P = \Pi_0\) and the first direct GLB.

%%% Extra-stabilised mCR scheme %%%%%%%%%%%%%%%%%%%%%%%%%%%%%%%%%%%%%%%%%%%%%%%%%%%%%%%%%%%%%%%%%%%%%%%%%%
\subsection{Extra-stabilised mCR and main result} \label{sec:5.1}
Consider the space \(\V_\es \coloneqq \V_\pw \times \V_\nc\) with
\begin{equation}
    \V_\nc \coloneqq \eCR_0^1(\Tcal) = \CR_0^1(\Tcal) \oplus B(\Tcal)
    \subset P_1(\Tcal) \oplus B(\Tcal) \eqqcolon \V_\pw.
    \label{eq:Ves}
\end{equation}
The extra-stabilised modified Crouzeix-Raviart (sCR) EVP seeks algebraic
eigenpairs \((\lambda_\es,u_\es) \in \R_+ \times \V_\es \setminus \{0\}\)
such that
\begin{equation}
    a_\es(u_\es,v_\es) = \lambda_\es \, b_\es(u_\es,v_\es)
    \quad\textup{for all } v_\es \in \V_\es.
    \label{eq:sCR}
\end{equation}
For \(u_\es = (u_\pw,u_\nc)\), \(v_\es = (v_\pw,v_\nc) \in \V_\es\), the
semi-scalar products in \eqref{eq:sCR} read
\begin{align*}
    a_\es(u_\es,v_\es)
    &=  a_\pw(u_\nc,v_\nc)
        + (\Pi_0 u_\nc, \Pi_0 v_\nc)_V
        + \kappa_\eCR^{-2} \big(h_\Tcal^{-2}(u_\pw-u_\nc), v_\pw-v_\nc\big)_\Omega, \\
    b_\es(u_\es,v_\es)
    &=  b(u_\pw,v_\pw).
\end{align*}
Since \((\V_\es,a_\es)\) is a Hilbert space and \(b_\es(\bullet,\bullet)\) is a semi-scalar
product with kernel \(\{0\} \times \eCR_0^1(\Tcal)\), the algebraic eigenvalue problem
\eqref{eq:sCR} has \(N = \dim\V_\pw = \dim P_1(\Tcal) + \dim B(\Tcal) = (n+2) |\Tcal|
\) finite and positive eigenvalues \(0 < \lambda_{\es(1)} \leq \dots \leq \lambda_{\es
(N)} <\infty\), while \(\lambda_{\es(N+1)} = \lambda_{\es(N+2)} = \dots = +\infty\). The
sCR EVP utilises more degrees of freedom than the mCR EVP, but Figures \ref{fig:Lshape_1_harm}
and \ref{fig:anderson} clearly demonstrate the necessity of an alternative to the post-processed
GLB from Section \ref{sec:3}. The new method \eqref{eq:sCR} directly computes GLB for any space
dimension \(n \geq 2\).

\begin{rem}
    The extra-stabilisation of CR and eCR EVP for the Schr\"odinger eigenvalue problem
    \eqref{eq:SEVP} is possible, but the upper bounds on \(\|PIu\|_V^2 - \|u\|_V^2\)
    from Subsection \ref{sec:4.4} prevent direct GLB. For the mCR EVP, however, \(\|
    \Pi_0 I_\eCR u\|_V^2 - \|u\|_V^2 \leq 0\) from Lemma \ref{lem:4.9} holds and leads
    to amazing direct GLB from the sCR EVP.
\end{rem}

\begin{thm}[\(V \in P_0(\Tcal)\)] \label{thm:dGLB}
    The \(k\)-th exact eigenvalue \(\lambda = \lambda_k\) in \eqref{eq:SEVP} and the \(k\)-th
    positive sCR eigenvalue \(\lambda_\es=\lambda_{\es(k)}\) in \eqref{eq:sCR} for the same
    number \(k \in \N\) satisfy for \(\varepsilon'' = \kappa_\eCR^2 h_{\max}^2\), \((\bullet)
    _+ \coloneqq \max\{0,\bullet\}\), and a piecewise constant potential \(V \in P_0(\Tcal)\)
    that
    \[
        \GLB_\es \coloneqq \frac{\lambda_\es}{1 + (\varepsilon''\lambda_\es-1)_+} \leq \lambda.
    \]
\end{thm}

\begin{proof}
    Let \(u_1, \dots, u_k\) denote the first \(k \in \N\) eigenfunctions of \eqref{eq:SEVP} and
    abbreviate the \(L^2\)-orthogonal projection onto \(\V_\pw\) by \(\Pi_\pw : \V \to \V_\pw\).
    \\
    \\
    \textbf{Case 1.} If \(\Pi_\pw u_1, \dots, \Pi_\pw u_k\) are linearly dependent, there exists
    some \(u\in E(k) \equiv \linh\{u_1, \dots, u_k\}\) such that \(\|u\| = 1\) but \(\Pi_\pw u =
    0\). These properties of \(u\), \(\V_\nc \subset \V_\pw\), and (I2) for \(I \coloneqq I_\eCR
    \) reveal
    \begin{equation}
        1
        =       \|u\|^2
        =       \|u-\Pi_\pw u\|^2
        \leq    \|u-Iu\|^2
        \leq    \varepsilon'' \, \Vvert{u-Iu}_\pw^2
        \leq    \varepsilon'' \lambda
        \label{eq:5.1}
    \end{equation}
    with the Pythagoras identity from (I1) and \(\Vvert{u}^2 \leq \lambda \) from \eqref{eq:S12}
    in the last step. A case distinction completes the proof in Case 1. If \(\lambda_\es \leq 
    \lambda\), then
    \[
        \GLB_\es
        =       \frac{\lambda_\es}{1 + (\varepsilon''\lambda_\es-1)_+}
        \leq    \lambda_\es
        \leq    \lambda.
    \]
    Else if \(\lambda < \lambda_\es\), then \eqref{eq:5.1} provides \(1 \leq \varepsilon'' \lambda
    < \varepsilon'' \lambda_\es\), whence \(1/\varepsilon'' \leq \lambda\) reveals
    \[
        \GLB_\es
        =       \frac{\lambda_\es}{1 + (\varepsilon''\lambda_\es-1)_+}
        =       \frac{\lambda_\es}{\varepsilon'' \lambda_\es}
        =       \frac{1}{\varepsilon''}
        \leq    \lambda.
    \]
    The Case 1 holds in particular for \(k \geq \dim\V_\pw+1\), where the limit \(\lambda_\es \to
    \infty\) leads to a valid guaranteed lower bound \(\GLB_\es = 1/\varepsilon'' \leq \lambda\).
    \\
    \\
    \textbf{Case 2.} Suppose that \(\Pi_\pw u_1, \dots, \Pi_\pw u_k\) are linearly independent
    and define \(E_\es(k) \coloneqq \linh\{(\Pi_\pw u_1,Iu_1), \dots, (\Pi_\pw u_k, Iu_k)\}\).
    The min-max principle for \eqref{eq:sCR} reveals
    \begin{equation}
        \lambda_\es
        \leq    \max_{v_\es \in E_\es(k) \setminus \{0\}}
                \frac{a_\es(v_\es,v_\es)}{b_\es(v_\es,v_\es)}.
        \label{eq:5.2}
    \end{equation}
    Select a maximizer \(v_\es = (\Pi_\pw u, Iu) \in E_\es(k) \setminus\{0\}\) with \(u
    \in E(k)\) and \(\|u\| = 1\). Inequality \eqref{eq:5.2} ensures that
    \begin{equation}
        \begin{aligned}
            \lambda_\es \, \|\Pi_\pw u\|^2
            &=      \lambda_\es \, b_\es(v_\es,v_\es) \leq a_\es(v_\es,v_\es) \\
            &=      \Vvert{Iu}_\pw^2 + \|\Pi_0 Iu\|_V^2 
                    + \kappa_\eCR^{-2} \, \|h_\Tcal^{-1} (\Pi_\pw - I) u\|^2.
        \end{aligned}
        \label{eq:5.3}
    \end{equation}
    A Pythagoras identity and (I2) reveal
    \begin{equation}
        \begin{aligned}
            \Vert h_\Tcal^{-1} (\Pi_\pw - I) u \Vert^2
            &=      \Vert h_\Tcal^{-1} (1 - I) u \Vert^2 - \Vert h_\Tcal^{-1} (1 - \Pi_\pw) u \Vert^2 \\
            &\leq   \Vert h_\Tcal^{-1} (1 - I) u \Vert^2 - h_{\max}^{-2} \Vert (1 - \Pi_\pw) u \Vert^2 \\
            &\leq   \kappa_\eCR^2 \vvvert (1-I)u \vvvert_\pw^2 - h_{\max}^{-2} \Vert (1 - \Pi_\pw) u \Vert^2.
        \end{aligned}
        \label{eq:5.4}
    \end{equation}
    The combination of \eqref{eq:5.3}, \eqref{eq:5.4}, and \(\varepsilon''
    = \kappa_\eCR^2 h_{\max}^2\) result in
    \[
        \lambda_\es \|\Pi_\pw u\|^2
        \leq    \Vvert{Iu}_\pw^2 + \|\Pi_0 Iu\|_V^2
                + \Vvert{u-Iu}_\pw^2 - \varepsilon''^{-1} \|u-\Pi_\pw u\|^2.
    \]
    The Pythagoras identity \(\Vvert{Iu}_\pw^2 = \Vvert{u}^2 - \Vvert{u-Iu}_\pw^2\) from (I1)
    and \(\Vvert{u}^2 \leq \lambda - \|u\|_V^2\) from \eqref{eq:S12} provide
    \begin{align*}
        \lambda_\es \|\Pi_\pw u\|^2
        &\leq   \lambda + \|\Pi_0 Iu\|_V^2 - \|u\|_V^2 - \kappa_\eCR^{-2} h_{\max}^{-2} \|u-\Pi_\pw u\|^2 \\
        &\leq   \lambda - \kappa_\eCR^{-2} h_{\max}^{-2} \|u-\Pi_\pw u\|^2
    \end{align*}
    with \(\|\Pi_0 I_\eCR u\|_V^2 - \|u\|_V^2 \leq 0\) from Lemma \ref{lem:4.9} in the last step. 
    The Pythagoras identity \(\|\Pi_\pw u\|^2 = \|u\|^2 - \|u-\Pi_\pw u\|^2\), \(\|u\|= 1\), and
    \(\|u-\Pi_\pw u\|^2 \leq \|u-Iu\|^2 \leq \varepsilon'' \Vvert{u-Iu}_\pw^2\) by \(\V_\nc \subset
    \V_\pw\) result in
    \begin{equation*}
        \lambda_\es - \lambda
        \leq    (\lambda_\es - \kappa_\eCR^{-2} h_{\max}^{-2})_+ \, \|u-\Pi_\pw u\|^2
        \leq    (\varepsilon'' \lambda_\es - 1)_+ \, \Vvert{u}^2.
    \end{equation*}
    Since \(\Vvert{u}^2 \leq \lambda\) from \eqref{eq:S12}, it holds \(\lambda_\es - \lambda \leq
    (\varepsilon'' \lambda_\es - 1)_+ \lambda\). A re-arrangement concludes the proof in Case 2.
\end{proof}

\begin{rem}[comparison with mCR on uniform meshes] \label{rem:comp}
    Direct calculations in the particular case that the mesh-size \(h_T\) equals \(h_{\max}\) for
    any simplex \(T \in\Tcal\) and for \(\varepsilon''\lambda_\es < 1\) provide that the \(k\)-th
    sCR eigenpair \((\lambda_\es, u_\es) \in \R_+ \times \V_\es\) with nonconforming component \(
    u_\nc \in \V_\nc=\eCR_0^1(\Tcal)\) in \(u_\es = (u_\pw,u_\nc)\) satisfies the rational EVP
    \begin{equation}
        a_\pw(u_\nc,v_\nc) + (\Pi_0u_\nc,\Pi_0v_\nc)_V
        =   \lambda_\es \, b\Big(\frac{u_\nc}{1 - \lambda_\es \kappa_\eCR^2 h_\Tcal^2},v_\nc\Big)
        \quad\textup{for all } v_\nc \in \V_\nc.
        \label{eq:rational}
    \end{equation}
    The details are analogous to \cite[Section 2.3]{CP23} and hence omitted. More important is
    the conclusion from this relation: In comparison of \eqref{eq:rational} with the mCR EVP
    \eqref{eq:dSEVP} we infer as in \cite[Theorem 4.6]{CZZ20} that \(\GLB_\es = \lambda_\es =
    \lambda_\mCR/(1 + \varepsilon''\lambda_\mCR) = \GLB_\mathrm{CECR}\). Remark \ref{rem:CECR}
    provides \(\GLB_\mathrm{CECR} < \GLB_\mCR\) in all experiments from Section \ref{sec:7}
    (except the coarsest mesh in Figure \ref{fig:Square_20_harm}). This explains that \(\GLB_
    \es\) is inferior to \(\GLB_\mCR\) for uniform mesh-refinement in the numerical examples
    of Section \ref{sec:7}.
\end{rem}

\begin{rem}[convergence sCR]
    The numerical experiments in Section \ref{sec:7} support the conjecture of optimal a-priori
    convergence rates for the sCR scheme \eqref{eq:sCR} as well as the optimal convergence of
    AFEM for the direct GLB in Theorem \ref{thm:dGLB}. Future research shall clarify whether
    the arguments from \cite{CP23,CP24} can be generalised to \eqref{eq:sCR} for the Schrödinger
    EVP \eqref{eq:SEVP}.
\end{rem}

%%% Scalar diffusion %%%%%%%%%%%%%%%%%%%%%%%%%%%%%%%%%%%%%%%%%%%%%%%%%%%%%%%%%%%%%%%%%%%%%%%%%%%%%%%%%%%%%
\subsection{Extension for piecewise constant diffusion} \label{sec:5.2}
Let \(\alpha \in P_0(\Tcal)\) be a positive piecewise constant diffusion coefficient with \(0 <
\underline{\alpha} \leq \alpha \leq \overline{\alpha} < \infty\) in the symmetric second-order
eigenvalue problem
\begin{equation}
    -\div(\alpha \nabla u) + Vu = \lambda u \textup{ in } \Omega
    \qquad\textup{and}\qquad
    u = 0 \textup{ on } \partial\Omega.
    \label{eq:GEVP}
\end{equation}
The weak form consider the weighted energy scalar product
\[
    a_\alpha(u,v) \coloneqq \int_\Omega \alpha \, \nabla u \cdot \nabla v \dx
\]
with induced norm \(\Vvert{\bullet}_\alpha \coloneqq a_\alpha(\bullet,\bullet)^{1/2}\). The enriched
Crouzeix-Raviart interpolation \(I \coloneqq I_\eCR\) satisfies the orthogonality (I1) verbatim for
\(a_\alpha(\bullet,\bullet)\), while the estimate (I2) becomes
\[
    \|u-Iu\|
    \leq    h_{\max} \kappa_\eCR \, \Vvert{u-Iu}_{\pw}
    \leq    h_{\max} \kappa_\eCR \underline{\alpha}^{-1/2} \, \Vvert{u-Iu}_{\alpha,\pw}.
\]
Replacing \(\kappa_\eCR\) by \(\kappa_\eCR \underline{\alpha}\) in the extra-stabilisation
\eqref{eq:sCR} leads to the following.

\begin{thm} \label{thm:5.1}
    The \(k\)-th exact eigenvalue \(\lambda = \lambda_k\) in \eqref{eq:GEVP} and the \(k\)-th
    sCR eigenvalue \(\lambda_\es=\lambda_{\es(k)}\) for the same number \(k \in \N\) satisfy
    for piecewise constant \(\alpha,V\in P_0(\Tcal)\) and \(\varepsilon'' = h_{\max}^2 \kappa
    _\eCR^2\) that
    \vspace*{-0.4cm} 
    \[
        \GLB \coloneqq \frac{\lambda_\es}{1 + (\varepsilon''\lambda_\es\underline{\alpha}^{-1}-1)_+}
        \leq \lambda.
    \]
\end{thm}
\begin{proof}
    The proof of Theorem \ref{thm:dGLB} in Subsection \ref{sec:5.1} applies verbatim, when \(
    \kappa_\eCR\,\underline{\alpha}^{-1/2}\) and \(\Vvert{\bullet}_\alpha\) replace \(\kappa_
    \eCR\) and \(\Vvert{\bullet}\).
\end{proof}

\begin{rem}[matrix-valued diffusion]
    General symmetric second-order eigenvalue problems allow a symmetric and (uniformly) elliptic
    diffusion coefficient \(\mathbf{A} \in L^\infty(\Omega;\mathbb{S})\) with the space of symmetric
    \(n \times n\) matrices \(\mathbb{S} \coloneqq \R_\mathrm{sym}^{n \times n}\) in
    \[
        -\div(\mathbf{A} \nabla u) + Vu = \lambda u \textup{ in } \Omega
        \qquad\textup{and}\qquad
        u = 0 \textup{ on } \partial\Omega.
    \]
    For a piecewise constant diffusion \(\mathbf{A} \in P_0(\Tcal;\mathbb{S})\), the orthogonality (I1)
    fails for the enriched Crouzeix-Raviart FEM, but holds for the generalized Crouzeix-Raviart scheme
    \cite{HM14}; see also \cite{Ain07,Voh07} for an equivalent post-processing. The remaining analysis 
    for the generalized Crouzeix-Raviart FEM is the same as in Subsections \ref{sec:5.1} and \ref{sec:5.2},
    provided that \(\underline{\alpha} > 0\) is a lower bound to the spectrum of \(\mathbf{A}\). Then
    the GLB in Theorem \ref{thm:5.1} follows verbatim.
\end{rem}

\begin{rem}[coefficients in \(L^\infty(\Omega)\)]
    The GLB in Theorems \ref{thm:GLB_eCR}, \ref{thm:GLB_RT}, \ref{thm:GLB_mCR} as well as
    the GLB in Theorems \ref{thm:dGLB} and \ref{thm:5.1} require \(V\in P_0(\Tcal)\). For
    general \(V\in L^\infty(\Omega;[0,\infty))\) perturbation arguments in the continuous
    and discrete problems are necessary \cite{Par98}; see also \cite[Theorem 4.3]{HM14}
    for GLB in a perturbation analysis for the diffusion \(A\).
\end{rem}
\vspace*{-0.2cm}

%%% Guaranteed upper eigenvalue bounds %%%%%%%%%%%%%%%%%%%%%%%%%%%%%%%%%%%%%%%%%%%%%%%%%%%%%%%%%%%%%%%%
\section{Guaranteed upper eigenvalue bounds} \label{sec:6}
\vspace*{-0.2cm} 
This section is devoted to guaranteed upper eigenvalue bounds (GUB) by averaging of computed
nonconforming approximations of eigenfunctions as in \cite{CG14,HHS15}.
\\
\\
Let \((u_1,\lambda_1), \dots, (u_K,\lambda_K)\) be the first \(K\) exact eigenpairs in
\eqref{eq:SEVP}, which define the subspace \(E(K)\coloneqq\mathrm{span}\{u_1,\dots,u_K
\}\) of dimension \(K \in \N\). Any nonconforming scheme above provides \(K\) discrete
eigenfunctions \(u_{h(1)},\dots,u_{h(K)} \in \Vh\) that span the vector space \(E_h(K)
\coloneqq\mathrm{span}\{u_{h(1)},\dots,u_{h(K)}\}\). The orthonormality of \(u_{h(1)},
\dots,u_{h(K)}\) ensures \(\dim E_h(K) = K \leq \dim\Vh\). Any linear operator \(\Acal
\) from \(\Vh\) into \(S_0^m(\Tcal) \coloneqq P_m(\Tcal) \cap \V \subset \V\) for some
\(m \in \N\) defines the conforming linear subspace \(\Acal E_h(K) \coloneqq \mathrm{
span}\{\Acal u_{h(1)}, \dots, \Acal u_{h(K)}\}\). The Rayleigh-Ritz min-max principle
\eqref{eq:minMax} suggests the following approximation of \(\lambda_k\) for all \(k
\in \{1,\dots,K\}\)
\begin{equation}
    \mu_k \coloneqq \min_{\substack{V_{\rm c} \subseteq \Acal E_h(K) \\ \dim V_{\rm c}=k}}
    \max_{v_{\rm c} \in V_{\rm c} \setminus\{0\}} \frac{a(v_{\rm c},v_{\rm c}) + (v_{\rm c
    },v_{\rm c})_V}{b(v_{\rm c},v_{\rm c})}.
    \label{eq:GUB}
\end{equation}
Define the matrices \(\texttt{A},\texttt{B} \in\R^{K \times K}\) by
\begin{equation}
    \begin{aligned}
        \texttt{A}_{jk}
        &=  a\big( \mathcal{A}u_{h(j)},\mathcal{A}u_{h(k)} \big)
            + \big( \mathcal{A}u_{h(j)},\mathcal{A}u_{h(k)} \big)_V, \\
        \texttt{B}_{jk}
        &=  b\big( \mathcal{A}u_{h(j)},\mathcal{A}u_{h(k)} \big)
    \end{aligned}
    \qquad\textup{for all } j,k = 1, \dots, K.
    \label{eq:AB}
\end{equation}
The \(k\)-th positive eigenvalue \(\mu_k\) of the algebraic eigenvalue problem \(\texttt{A}x =
\mu \texttt{B}x\) satisfies \(\mu_k \geq \lambda_k\). If this algebraic eigenvalue problem has
less than \(k\) positive eigenvalues (counting multiplicities), then \(\dim\Acal E_h(K) \leq k
-1\). The latter is possible, as linear dependence of \(\Acal u_{h(1)},\dots, \Acal u_{h(K)}\)
leads to a Rayleigh quotient \(0/0\) in \eqref{eq:GUB} and MATLAB \(\mu=\texttt{sort(eig(A,B))%
}\) outputs \texttt{NaN}.

\begin{rem}[injective conforming companion] \label{rem:6.1}
    If \(\Acal: \Acal E_h(K) \to S_0^m(\Tcal)\) is injective, then \(\dim E_h(K) = K\) and hence
    the \(k\)-th positive eigenvalue \(\mu_k\) in \(\texttt{A}x = \mu\texttt{B}x\) exists for \(
    k \leq K\) and satisfies \(\mu_k \geq \lambda_k\).
\end{rem}

\begin{ex}[\(P_1\) averaging for CR] \label{ex:P1}
    At any interior vertex \(z \in \Vcal(\Omega)\) with patch \(\Tcal(z) \coloneqq \{T \in \Tcal
    : z \in \Vcal(T)\}\), \(v_\CR \in \CR_0^1(\Tcal)\) has in general different values \(v_\CR|_
    T(z)\) at \(T \in \Tcal(z)\). Their arithmetic mean defines the linear operator \(\Acal_1 :
    \CR_0^1(\Tcal) \to S_0^1(\Tcal)\) at any \(v_\CR \in \CR_0^1(\Tcal)\) by
    \[
        (\Acal_1 v_\CR)(z) \coloneqq
        \begin{cases}
            \vert\Tcal(z)\vert^{-1} \sum_{T \in \Tcal(z)} v_\CR|_T(z) &\textup{for } z \in \Vcal(\Omega), \\
            0   &\textup{for } z \in \Vcal(\partial\Omega).
        \end{cases}
    \]
    Since \(\dim S_0^1(\Tcal) = |\Vcal(\Omega)| \leq |\Fcal(\Omega)|-1 = \dim\CR_0^1(\Tcal)-1\)
    if \(\Tcal\) consists of more than one simplex, \(\Acal_1\) is \emph{not} injective. However,
    it suffices to solve the small algebraic EVP \(\texttt{A}x = \mu\texttt{B}x\) to check if a
    positive eigenvalue \(\mu_k\) exists, which provides \(\mu_k \geq \lambda_k\).
\end{ex}

\begin{ex}[\(P_2\) averaging for CR] \label{ex:P2}
    Consider the operator \(\Acal_2 : \CR_0^1(\Tcal)\to S_0^2(\Tcal)\) that maps any \(v_\CR
    \in \CR_0^1( \Tcal)\) to \(\Acal_2 v_\CR \in S_0^2(\Tcal)\) by \(\Acal_2 v_\CR \coloneqq
    \Acal_1 v_\CR\) at \(\Vcal\) and
    \[
        (\Acal_2 v_\CR)(\mid(F)) \coloneqq v_\CR(\mid(F)) \qquad\textup{for all } F \in \Fcal.
    \]
    Since any \(v_\CR \in\CR_0^1(\Tcal)\setminus\{0\}\) satisfies \(v_\CR(\mid(F)) \neq 0\)
    for some \(F \in \Fcal(\Omega)\), we have \(\Acal_2 v_\CR \neq 0\). Thus \(\Acal_2\) is
    injective. Remark \ref{rem:6.1} ensures that \(\mu_k \geq \lambda_k\) exists and is a
    guaranteed upper bound for any \(k \leq K\).
\end{ex}

\begin{ex}[averaging for eCR] \label{ex:P1eCR}
    The nodal averaging operator \(\Acal_1 : \CR_0^1(\Tcal) \to S_0^1(\Tcal)\) of Example \ref{ex:P1}
    and the Crouzeix-Raviart interpolation \(I_\CR : \V + \eCR_0^1(\Tcal) \to \CR_0^1(\Tcal)\) define
    the averaging operator \(\Acal_1 \circ I_\CR : \eCR_0^1(\Tcal) \to S_0^1(\Tcal)\) by composition.
    Since \(\dim S_0^1(\Tcal) - \dim\eCR_0^1(\Tcal) = 1 - 2|\Tcal| < 0\) (for a simply connected domain),
    the operator \(\Acal_1 \circ I_\CR\) is \emph{not} injective.
\end{ex}

\noindent
The theorem below asserts a sufficient criterion for the existence of a \(k\)-th positive eigenvalue
\(\mu_k \geq \lambda_k\) in \(\texttt{A}x = \mu \texttt{B}x\).

\begin{thm} \label{thm:6.8}
    If \(\alpha_k \coloneqq \sup_{v_h \in E_h(k)\setminus\{0\}} \|v_h-\Acal v_h\|/\Vvert{v_h}_\pw <
    \lambda_{h(k)}^{-1/2}\), then \(\mu_k \geq \lambda_k\).
\end{thm}
\begin{proof}
    Any \(v_h \in E_h(k)\) with \(\|v_h\| = 1\) satisfies \(\Vvert{v_h}_\pw^2  \leq \lambda_{h(k)}\)
    from \(v_h \in E_h(k)\),~whence
    \[
        \| v_h - \Acal v_h \| \leq \alpha_k \, \Vvert{v_h}_\pw \leq \alpha_k \, \lambda_{h(k)}^{1/2}.
    \]
    A reverse triangle inequality, \(\|v_h\| = 1\), and the assumption \(\alpha_k < \lambda_{h(k)}^{
    -1/2}\) reveal
    \[
        0 < 1 - \alpha_k \lambda_{h(k)}^{1/2} \leq 1 - \|v_h-\Acal v_h\| \leq \|\Acal v_h\|.
    \]
    Consequently, \(\Acal|_{E_h(k)}\) is injective and Remark \ref{rem:6.1} implies \(\dim\Acal
    E_h(k) = k\). Thus, \(\mu_k \geq \lambda_k\) holds and concludes the proof.
\end{proof}

\begin{rem}[mesh-size condition]
    For the averaging operator of Example \ref{ex:P1eCR}, \cite{HM14} provides an upper bound \(\alpha
    _k\leq \beta h_{\max}\) with some explicit mesh-dependent constant \(\beta>0\). Hence, Theorem \ref{thm:6.8} results in the a-priori mesh-size condition \(h_{\max}< 1/\beta\lambda_{h(k)}^{1
    /2}\) sufficient for the existence of \(\mu_k \geq \lambda_k\). The authors in \cite{HM14} remark
    that this condition is \emph{not} strict. This is due to the fact that \(\beta h_{\max}\)
    is in fact an upper bound on the supremum over \(\Vh\setminus\{0\}\) rather than \(E_h(k) \setminus
    \{0\}\) in the definition of \(\alpha_k\).
\end{rem}

\begin{rem}[compute \(\alpha_k\)]
    Define \(\texttt{C},\texttt{D} \in \R^{k \times k}\) by \(\texttt{D} \coloneqq \mathrm{diag}(
    \lambda_{h(1)},\dots,\lambda_{h(k)})\) and
    \[
        \texttt{C}_{j\ell} \coloneqq b\big( u_{h(j)}-\Acal u_{h(j)}, u_{h(\ell)}-\Acal u_{h(\ell)}
        \big) \qquad\textup{for } j,\ell = 1, \dots, k.
    \]
    The largest eigenvalue \(\gamma_{\max}\) in the algebraic EVP \(\texttt{C}x = \gamma\texttt{D}
    x\) provides \(\alpha_k = \gamma_{\max}^{1/2}\).
\end{rem}

\begin{rem}[a-priori analysis]
    Suppose that the continuos eigenfunction \(u\) in \eqref{eq:SEVP} satisfies \(u\in\V\cap
    H^{1+s}(\Omega)\) for some \(0<s\leq1\). Then the a-priori analysis in \cite{HHS15,HM14}
    reveals for the averaging from Example \ref{ex:P1eCR} that \(\mu_k-\lambda_k=\mathcal{O}
    (h_{\max}^{2s})\) for the harmonic EVP. It is expected that the additional \(L^2\) term
    in \eqref{eq:SEVP} does \emph{not} obstruct the arguments therein. Note that the proof
    of the a-priori convergence rates simplifies for the averaging from Example \ref{ex:P1}.
\end{rem}

%%% Numerical Experiments %%%
\section{Numerical experiments} \label{sec:7}
\vspace{-0.2cm} 
Computational benchmarks compare GLB and GUB from the schemes of Table \ref{tab:1} for different
potentials \(V\) in Figure~\ref{fig:groundStates} on two convex squares and the non-convex L-shaped
domain of Figures \ref{fig:inital_mesh} and \ref{fig:anderson_mesh}.

\begin{figure}[ht]
	\centering
	\scalebox{0.3}{\includegraphics{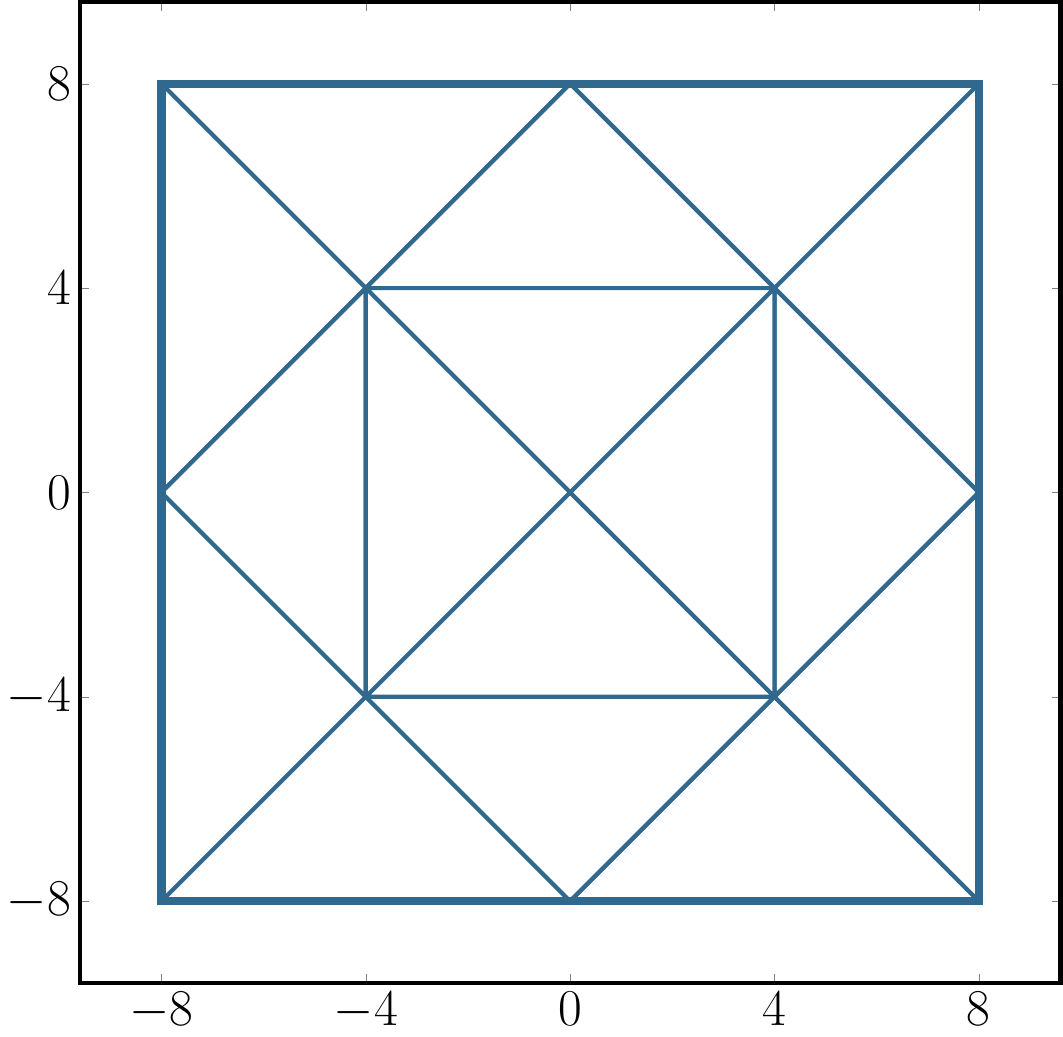}}
	\hspace{1cm}
	\scalebox{0.3}{\includegraphics{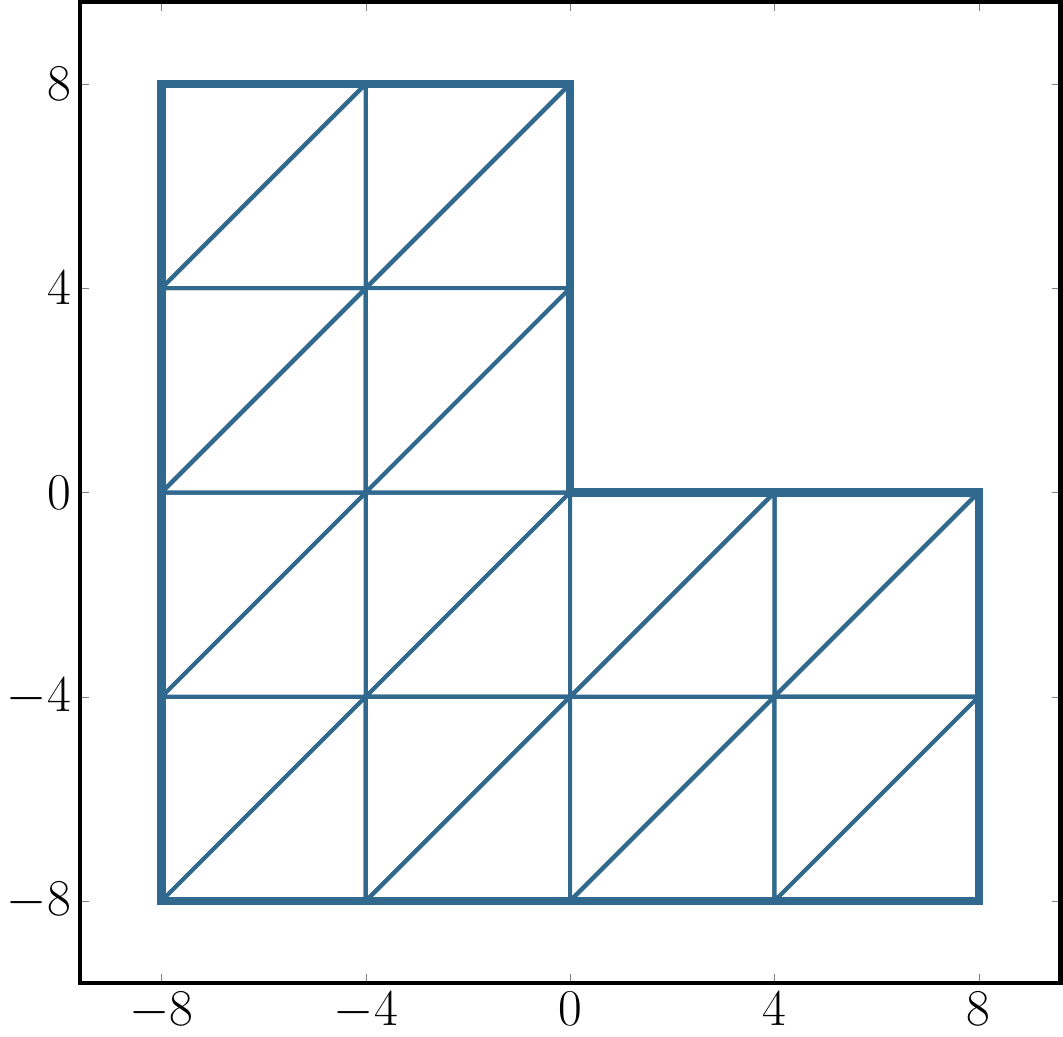}}
	\vspace{-0.1cm} 
	\caption{Initial triangulations of the square (left) and L-shaped domain (right) in Subsection
		\ref{sec:7.2} and \ref{sec:7.3}.}
	\label{fig:inital_mesh}
	\vspace{-0.2cm} 
\end{figure}

%% Remarks on the implementation %%%
\subsection{Remarks on the implementation} \label{sec:7.1}
The MATLAB implementation extends \cite{ACS99,BC05} and solves all algebraic eigenvalue problems
with \texttt{eigs}. Round-off error are expected very small and hence neglected for simplicity.
The convergence history plots display the non-negative accuracy \(\lambda_k-\GLB(k)\) or \(\GUB
(k)-\lambda_k\) for the \(k\)-th eigenvalue \(\lambda_k\) in \eqref{eq:SEVP} and the GLB/GUB in
Figure~\ref{fig:GLB_legend} over the number of triangles \(|\Tcal|\) for uniform red-refinement
\(\theta = 1\) and adaptive mesh-refinement \(\theta = 0.5\) in the D\"orfler marking algorithm
\cite{Do96}. The exact eigenvalue \(\lambda_k\) is unknown in all examples below. Aitken extrapolations
from conforming Courant (S1) eigenvalue approximations on very fine meshes provide all reference
values for \(\lambda_k\).

\begin{figure}[ht]
    \centering
    \includegraphics[width=0.75\linewidth]{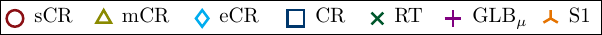}
    \vspace{-0.1cm} 
    \caption{Methods for GLB/GUB in the computational benchmarks of Section \ref{sec:7}.}
    \label{fig:GLB_legend}
\end{figure}

\noindent
Given a discrete eigenpair \((\lambda_h,u_h)\in\R_+\times\Vh\), the adaptive mesh-refinement
\cite{Do96,Car+14,CP24} with newest-vertex bisection is driven by the (heuristic) refinement
indicator \(\eta^2(T) \coloneqq \eta(T)^2\) defined for any \(T \in \Tcal\) by
\begin{equation}
    \eta^2(T) \coloneqq |T| \, \|(\lambda_h-V)u_\CR\|_{L^2(T)}^2 + |T|^{1/2} \sum_{F \in \Fcal(T)}
    \| [\nabla u_\CR]_F \times \nu_F \|_{L^2(F)}^2
    \label{eq:est}
\end{equation}
for the tangential component \([\nabla u_\CR]_F \times \nu_F\) of the jump \([\nabla u_\CR]_F\)
across \(F\in\Fcal\). The estimator from \eqref{eq:est} is motivated from the source problem as
in \cite[Lemma 3.1 of Subsection 3.2.2]{CGN24}. The selection of \(u_\CR \in\CR_0^1(\Tcal)\) in
\eqref{eq:est} depends on the discrete EVP. For the CR EVP, \(u_\CR\) in \eqref{eq:est} is the
CR eigenfunction. Recall that any \(u_\eCR \in \eCR_0^1(\Tcal)\) decomposes uniquely as \(u_\eCR
= u_\CR + u_2\) with \(u_\CR \in \CR_0^1(\Tcal)\) and \(u_2 \in B(\Tcal)\). Hence, \(u_\CR\) in
\eqref{eq:est} is the CR part of the eCR and mCR eigenfunctions in the respective EVP. The sCR
eigenfunction reads \(u_\es = (u_\pw,u_\nc) \in \V_\pw \times \V_\nc\) and \(u_\CR\) is the CR
part of \(u_\nc\in\eCR_0^1(\Tcal)\) in \eqref{eq:est}. Undisplayed numerical experiments confirm
that individual refinement indicators (that is \eqref{eq:est} for each \(u_\CR\) describes above)
lead to outputs in the adaptive algorithm that do not differ significantly from each other.
Therefore, the adaptive mesh-refinement is driven by sCR EVP in \eqref{eq:est} in all experiments
below. The GUB for the \(k\)-th eigenvalue \(\lambda_k\) in \eqref{eq:SEVP} are computed based on
\(k\) nonconforming eigenfunctions \(u_{h(1)},\dots,u_{h(k)}\in\Vh\) with the averaging operators
\(\Acal\) from Example \ref{ex:P1} and \ref{ex:P1eCR}. For sCR EVP, \(\Acal\) applies to the nonconforming 
component \(u_\nc\in\eCR_0^1(\Tcal)\) of \(u_\es = (u_\pw,u_\nc) \in \V_\pw \times \V_\nc\). All
integrals in Subsection \ref{sec:7.3} are evaluated by a quadrature rule that is exact for
polynomials up to (total) degree at most 10.

%%% Harmonic potential %%%
\subsection{Harmonic potential} \label{sec:7.2}
We approximate the harmonic potential \(V_1(x) \coloneqq |x|^2/2\) on the square and L-shaped domain
from Figure~\ref{fig:inital_mesh} by piecewise constants \(\Pi_0 V_1\) and exploit the edge midpoint
quadrature rule for quadratic functions on a triangle \(T \in \Tcal\)
\[
    \Pi_0 V_1|_T
    =   \dashint_T \frac{|x|^2}{2} \dx
    =   \frac{1}{6} \sum_{F \in \Fcal(T)} |\mid(F)|^2.
\]

%%% Ground state %%%
\paragraph{Ground state on \(\mathbf{\Omega = (-8,8)^2}\).}
Figure~\ref{fig:Square_1_harm} displays the convergence history plot of the five GLB from Table \ref{tab:1}
and the GUB from CR, eCR, mCR, sCR, and the Courant EVP approximation of the principal eigenvalue \(\lambda
_1 = \sqrt{2}\) (conjectured based on undisplayed numerical evidence). Uniform mesh-refinement leads to the
empirical optimal convergence rate one for any scheme under investigation. The mCR GLB are most accurate and
significantly sharper than \(\GLB_\es\) that are slightly sharper than \(\GLB_\RT\). The GLB from eCR and CR
EVP show a pre-asymptotic regime without systematic convergence for triangulations with \(|\Tcal| \leq 10^3
\). Thereafter, the GLB converge with the optimal rate one, but \(\lambda_1-\GLB_\CR(1)\) is larger than
\(\lambda_1 - \GLB_\mCR(1)\) in Figure~\ref{fig:Square_1_harm} by more than two orders of magnitude even
on the finest mesh. This is expected due to the direct dependence of \(\GLB_\CR\) and \(\GLB_\eCR\) on \(
\|V_1\|_\infty = 64\). The two coarsest triangulations provide rough GUB approximations from mCR and sCR,
however, the strongly improved GUB match the accuracy of the Courant EVP for any finer mesh in Figure~\ref{%
fig:Square_1_harm}. The GUB from CR and eCR EVP are as accurate as the Courant EVP approximation on any
triangulation in Figure~\ref{fig:Square_1_harm} with less overall computational cost.

\begin{figure}[t]
	\centering
	\scalebox{0.72}{\includegraphics{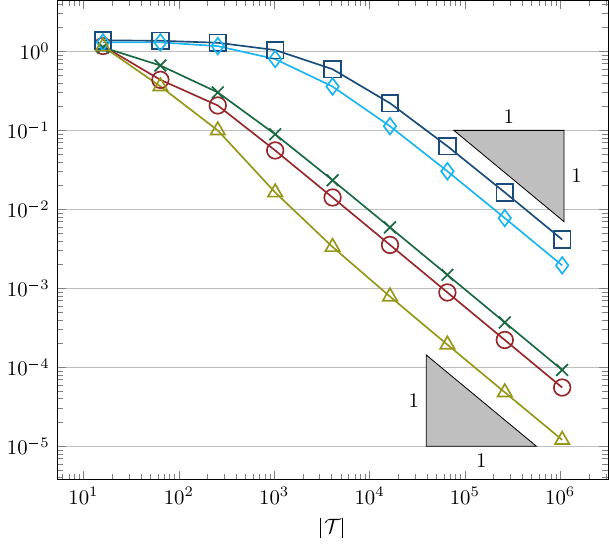}}
	\scalebox{0.72}{\includegraphics{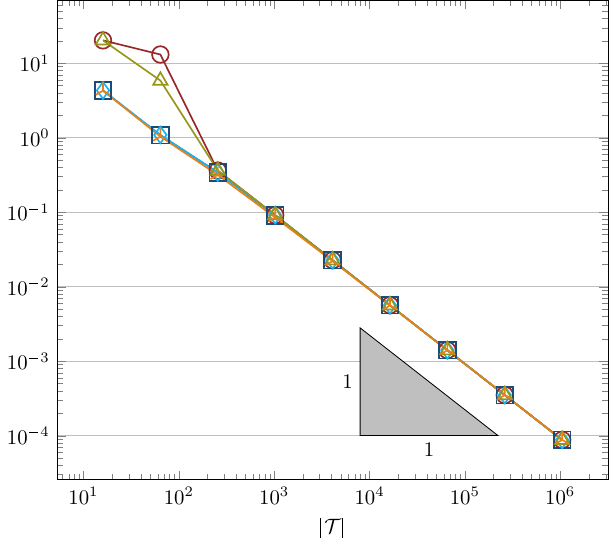}}
	\caption{Convergence history plot for \(\lambda_1-\GLB(1)\) (left) and \(\GUB(1)-\lambda_1\) (right)
		on uniform meshes of \(\Omega = (-8,8)^2\) for \(V_1\).}
	\label{fig:Square_1_harm}
\end{figure}

%%% Excited states %%%
\paragraph{Excited states on \(\mathbf{\Omega = (-8,8)^2}\).}
Figure~\ref{fig:Square_20_harm} displays the convergence history plot of \(\GLB_\mu(
20)\) from Theorem \ref{thm:GLB_Liu} in addition to that of the five GLB from Table
\ref{tab:1}. A separate plot in Figure~\ref{fig:Square_20_harm} shows the convergence
history of the GUB from the CR, eCR, mCR, and sCR EVP in comparison to the Courant EVP
approximation. Uniform red-refinement leads in Figure~\ref{fig:Square_20_harm} to the
optimal convergence rate one for all GLB and GUB towards \(\lambda_{20}= 8.4852765\).
The mCR GLB are most accurate followed by \(\GLB_\es\) and \(\GLB_\RT\). Remark \ref{rem:comp}
predicts \(\GLB_\es < \GLB_\mCR\) for small mesh-sizes on uniform meshes and this is
visible in Figure~\ref{fig:Square_20_harm}. The quotient \(\lambda_{20}/\lambda_1 >6
\) leads to \(\GLB_\mu(20) < \GLB_\CR(20)\), whence \(\GLB_\mu\) is the worst lower
bound in Figure~\ref{fig:Square_20_harm}: On the finest triangulation, \(\lambda_{20
}-\GLB_\mu(20)\) is still two orders of magnitude larger than \(\lambda_{20} - \GLB_
\mCR(20)\) and one order of magnitude larger than \(\GLB_\CR\) of Theorem \ref{thm:GLB_CR}.
Averaging of the sCR eigenfunctions leads to linear dependences and \texttt{NaN} in
\eqref{eq:GUB} on the three coarse triangulations, which are omitted in Figure~\ref{fig:Square_20_harm}.
On triangulations with \(|\Tcal| \geq 10^4\), however, \(\GUB_\es(20)\) matches the
accuracy of the Courant EVP approximation and any other GUB under investigation.

\begin{figure}[t]
	\centering
	\scalebox{0.72}{\includegraphics{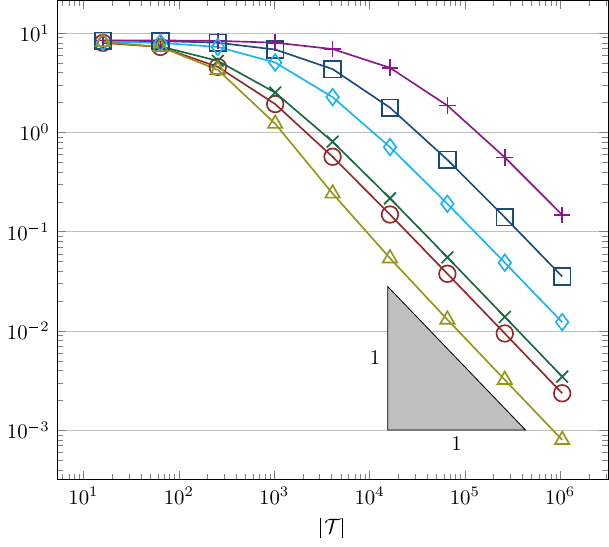}}
	\scalebox{0.72}{\includegraphics{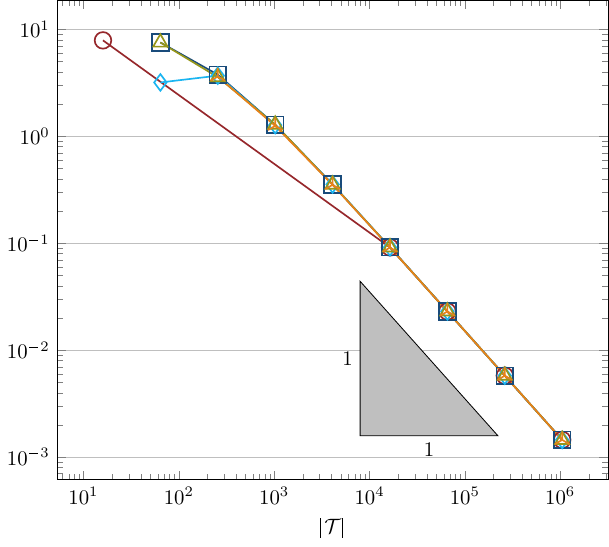}}
	\caption{Convergence history plot for \(\lambda_{20}-\GLB(20)\) (left) and \(\GUB(20)-\lambda_{20}\)
		(right) on uniform meshes of \(\Omega = (-8,8)^2\) for \(V_1\).}
	\label{fig:Square_20_harm}
\end{figure}

%%% L-shaped benchmark %%%
\paragraph{L-shaped benchmark.}
The re-entrant corner of the L-shaped domain in Figure~\ref{fig:inital_mesh} leads to the principal
eigenvalue \(\lambda_1 = 2.357076\) for the ground state \(u_1\in H_0^1(\Omega)\setminus H^2(\Omega
)\). This predicts a reduced convergence rate \(2/3\), which is visible for the guaranteed lower and
upper eigenvalue bounds in Figure~\ref{fig:Lshape_1_harm}. Adaptive mesh-refinement remarkably
recovers the optimal convergence rate one for all GUB in Figure~\ref{fig:Lshape_1_harm}. Figure
\ref{fig:adaptive_mesh_Lshape} depicts the expected local mesh-refinement at the origin. Figure
\ref{fig:Lshape_1_harm} shows that any post-processed GLB from Table \ref{tab:1} is almost constantly
equal to \(0\) as they suffer from a maximal mesh-size \(h_{\max} = \sqrt{8}\) in Figure
\ref{fig:adaptive_mesh_Lshape}: Adaptive meshes lead to useless post-processed GLB. Only the adaptive
sCR EVP displays optimal convergence rates and significantly improves any GLB on uniform meshes.

\begin{figure}[ht]
	\centering
	\scalebox{0.7}{\includegraphics{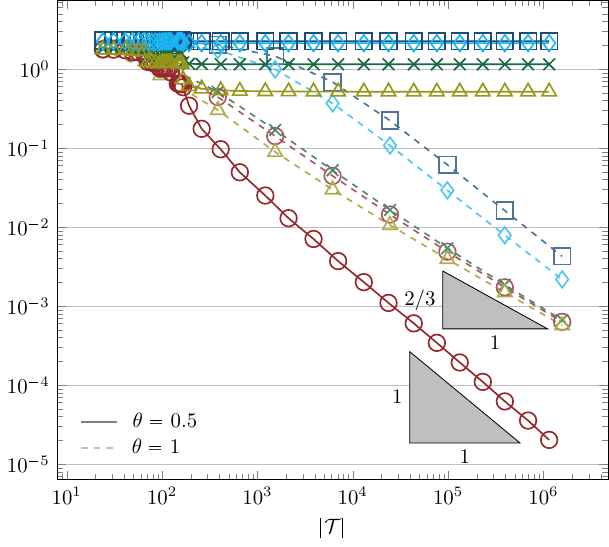}}
	\scalebox{0.7}{\includegraphics{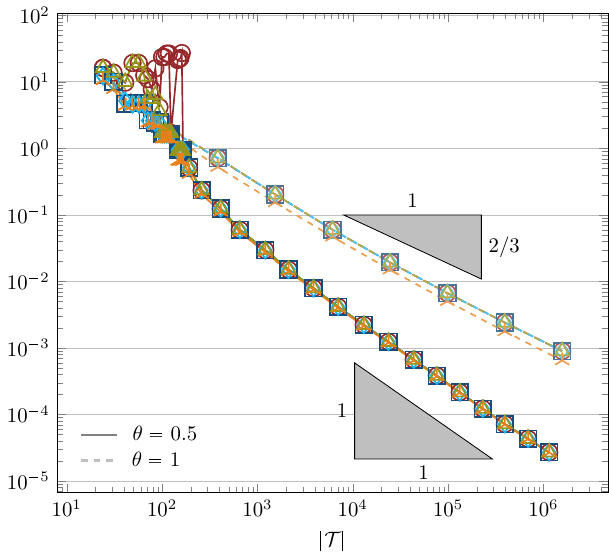}}
	\caption{Convergence history plot for \(\lambda_1-\GLB(1)\) (left) and \(\GUB(1)-\lambda_1\) (right)
		on uniform (dashed) and adaptive (solid) meshes of \(\Omega = (-8,8)^2 \setminus [0,8)^2\) for \(V_1\).}
	\label{fig:Lshape_1_harm}
\end{figure}

\begin{figure}[ht]
	\centering
	\scalebox{0.3}{\includegraphics{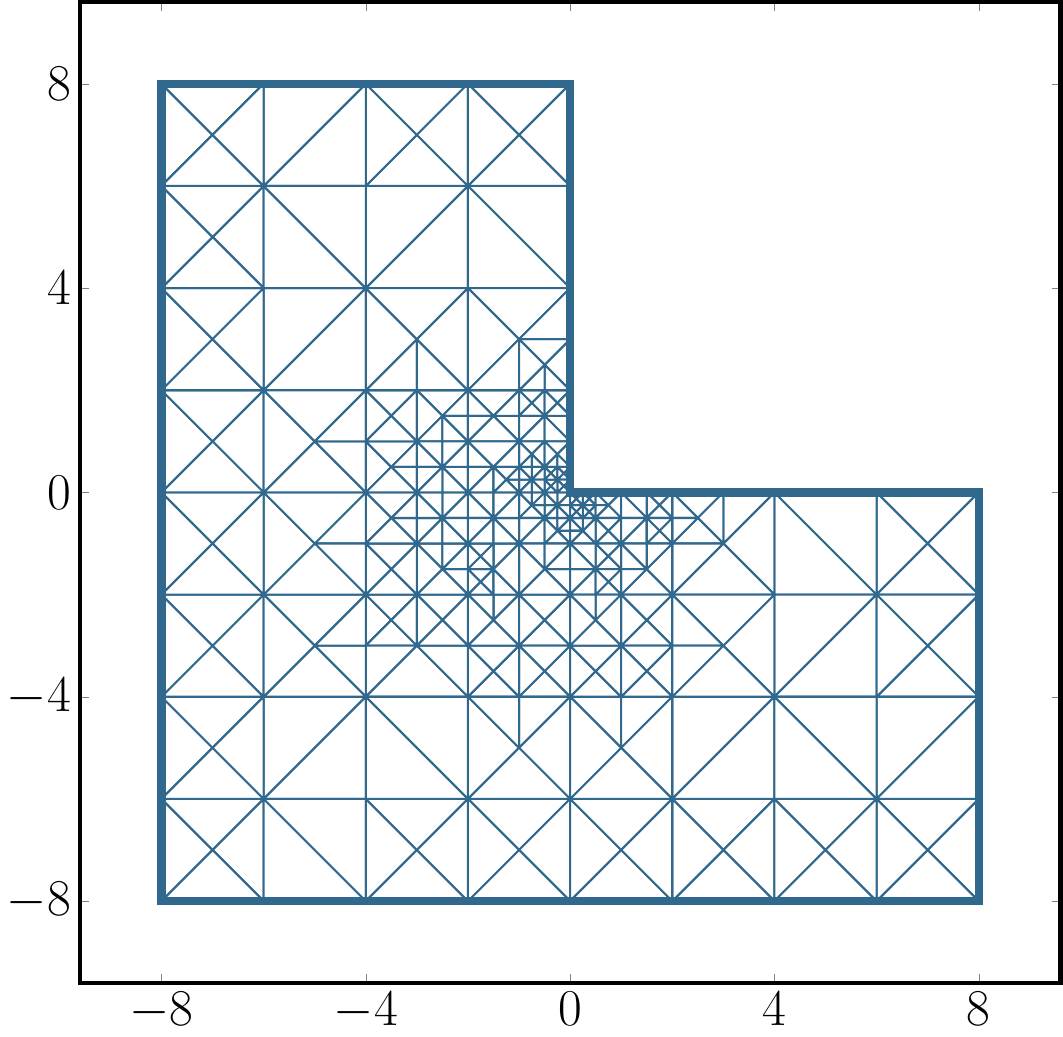}}
	\hspace{1cm}
	\scalebox{0.3}{\includegraphics{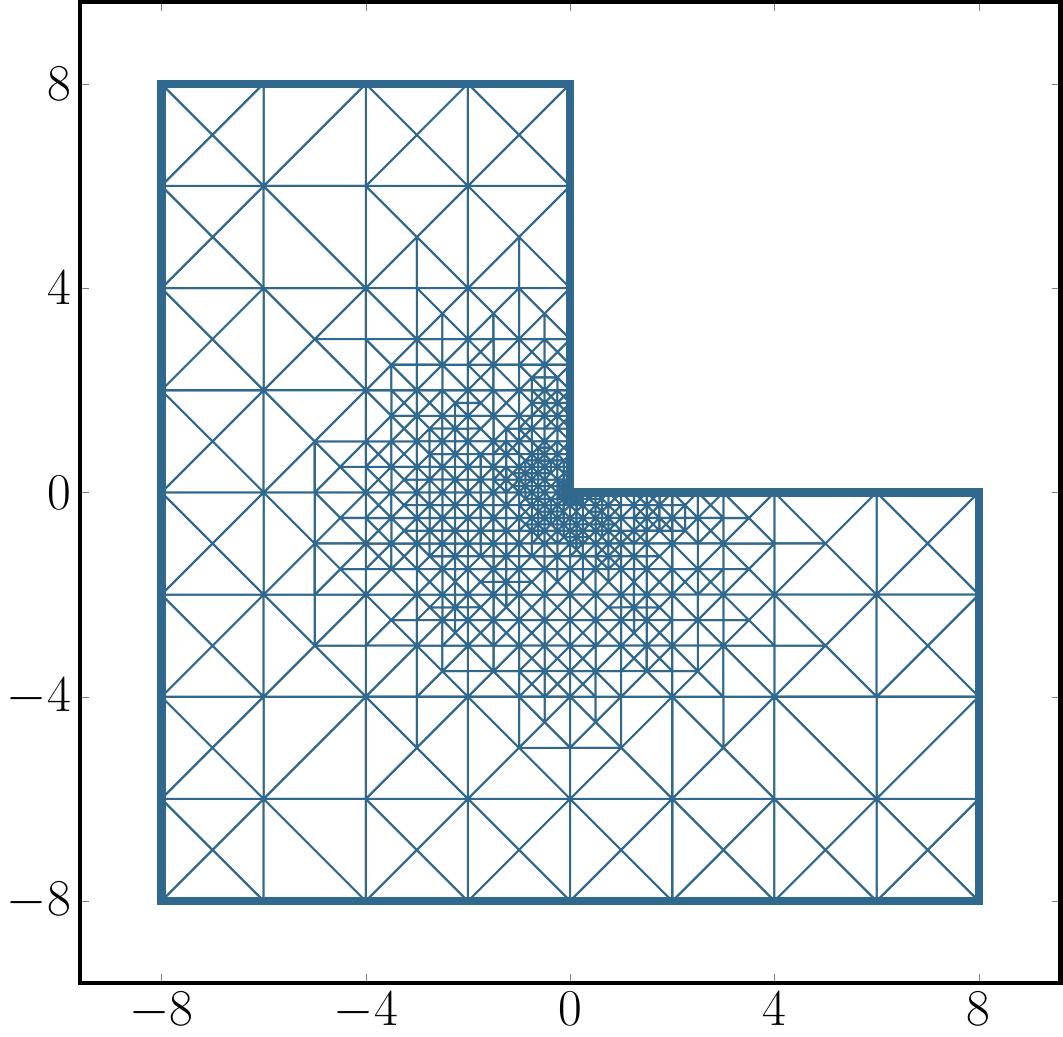}}
	\caption{Adaptive triangulations \(\Tcal_\ell\) on level \(\ell = 23\) (left) and \(\ell = 25\)
		(right) for \(\lambda_1\) with \(V_1\) on \(\Omega = (-8,8)^2 \setminus [0,8)^2\) with \(|\Tcal
		_{23}| = 412\) and \(|\Tcal_{25}| = 1274\) in Subsection \ref{sec:7.2}.}
	\label{fig:adaptive_mesh_Lshape}
\end{figure}

%%% Lattice potential %%%
\subsection{Lattice potential} \label{sec:7.3}
The lattice potential \(V_2(x) \coloneqq (|x|^2/2 - 16)_+ + \lfloor 30 + 10 \sin(\pi x_1/2) \sin(
\pi x_2/2)\rfloor\) in Figure~\ref{fig:groundStates} is adopted from \cite{HP17} and approximated
by piecewise constants \(\Pi_0V_2\). The quadrature error for the discontinuous function \(V_2\)
is neglected and results in perturbations of Rayleigh quotients \cite{CG14,Par98}.
\begin{figure}[t]
	\vspace*{-0.1cm} 
    \centering
    \scalebox{0.72}{\includegraphics{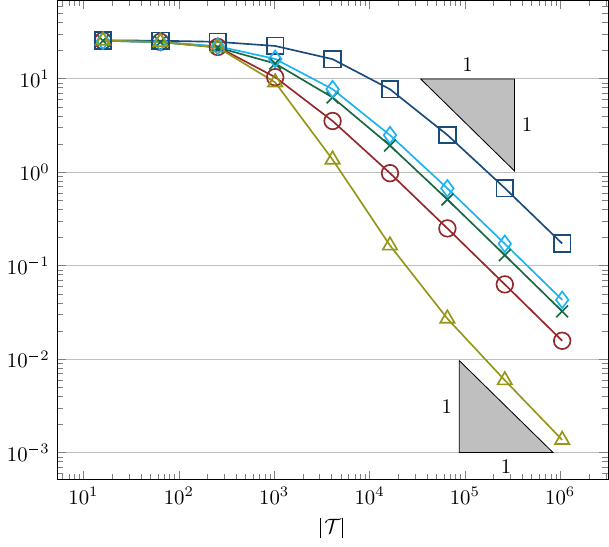}}
    \scalebox{0.72}{\includegraphics{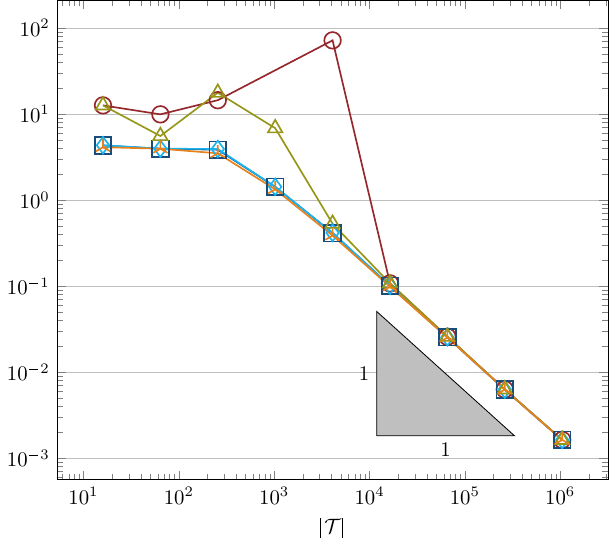}}
    \vspace{-0.3cm} 
    \caption{Convergence history plot for \(\lambda_1-\GLB(1)\) (left) and \(\GUB(1)-\lambda_1\) (right)
    on uniform meshes of \(\Omega = (-8,8)^2\) for \(V_2\).}
    \label{fig:Square_1_lattice}
    \vspace*{-0.3cm} 
\end{figure}
Figure~\ref{fig:Square_1_lattice} displays the convergence history plot of the GLB
from Table \ref{tab:1} and the GUB from CR, eCR, mCR, and sCR EVP as well as the Courant EVP
for \(\lambda_1 = 25.743622\) on \(\Omega = (-8,8)^2\). Figure~\ref{fig:Square_1_lattice}
shows pre-asymptotic stagnation on triangulations with \(|\Tcal|\leq 10^3\) for all GLB. This
observation applies verbatim for the GUB from CR, eCR, and Courant EVP on the three coarsest
triangulations in Figure~\ref{fig:Square_1_lattice}. Uniform mesh-refinement leads to the optimal
convergence rate one for any quantity under investigation. The mCR GLB are sharpest followed
by \(\GLB_\es\). The accuracy of the latter is one order of magnitude worse. This observation
agrees with Remark \ref{rem:comp} that predicts \(\GLB_\es <\GLB_\mCR\) for small mesh-sizes.
It is remarkable that \(\GLB_\eCR\) is only slightly inferior to \(\GLB_\RT\) although the former directly
depends on \(\|V_2\|_\infty = 88\). The CR GLB is the worst lower bound in Figure~\ref{fig:Square_1_lattice}
and less accurate than \(\GLB_\mCR\) by two orders of magnitude. The triangulation with \(|\Tcal|\leq 10^4\) 
provide coarse GUB from mCR and sCR EVP in Figure~\ref{fig:Square_1_lattice}, however, for finer
triangulation any scheme leads to GUB of identical accuracy. It is noteworthy that CR and eCR EVP
lead to GUB with the same accuracy as the Courant EVP on every mesh in Figure~\ref{fig:Square_1_lattice}.

\vspace{-0.1cm} 
\begin{figure}[ht]
    \centering
    \begin{subfigure}[t]{0.32\textwidth}
        \centering
        \scalebox{0.38}{\includegraphics{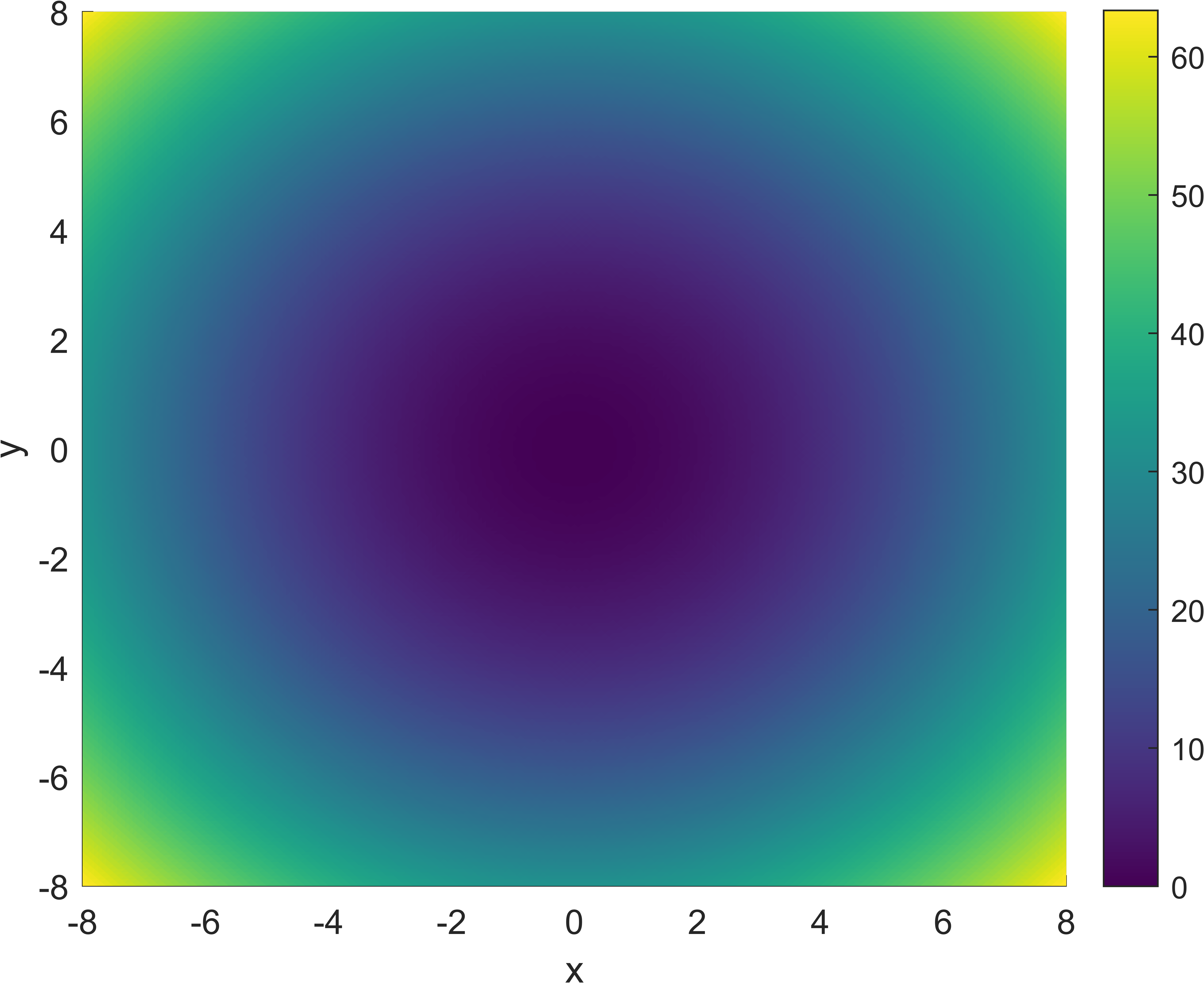}}
    \end{subfigure}
    \begin{subfigure}[t]{0.32\textwidth}
        \centering
        \scalebox{0.38}{\includegraphics{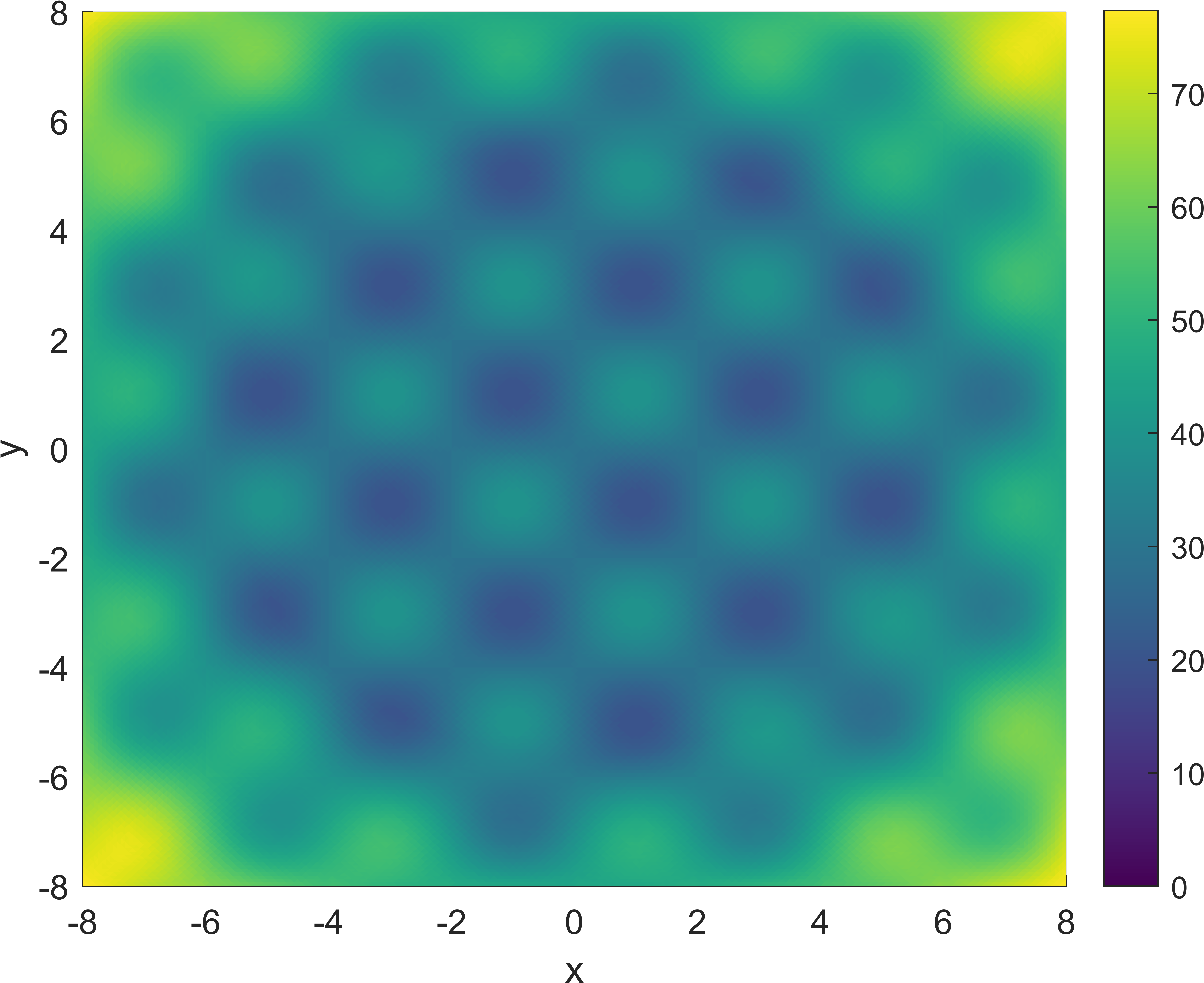}}
    \end{subfigure}
    \begin{subfigure}[t]{0.32\textwidth}
        \centering
        \scalebox{0.38}{\includegraphics{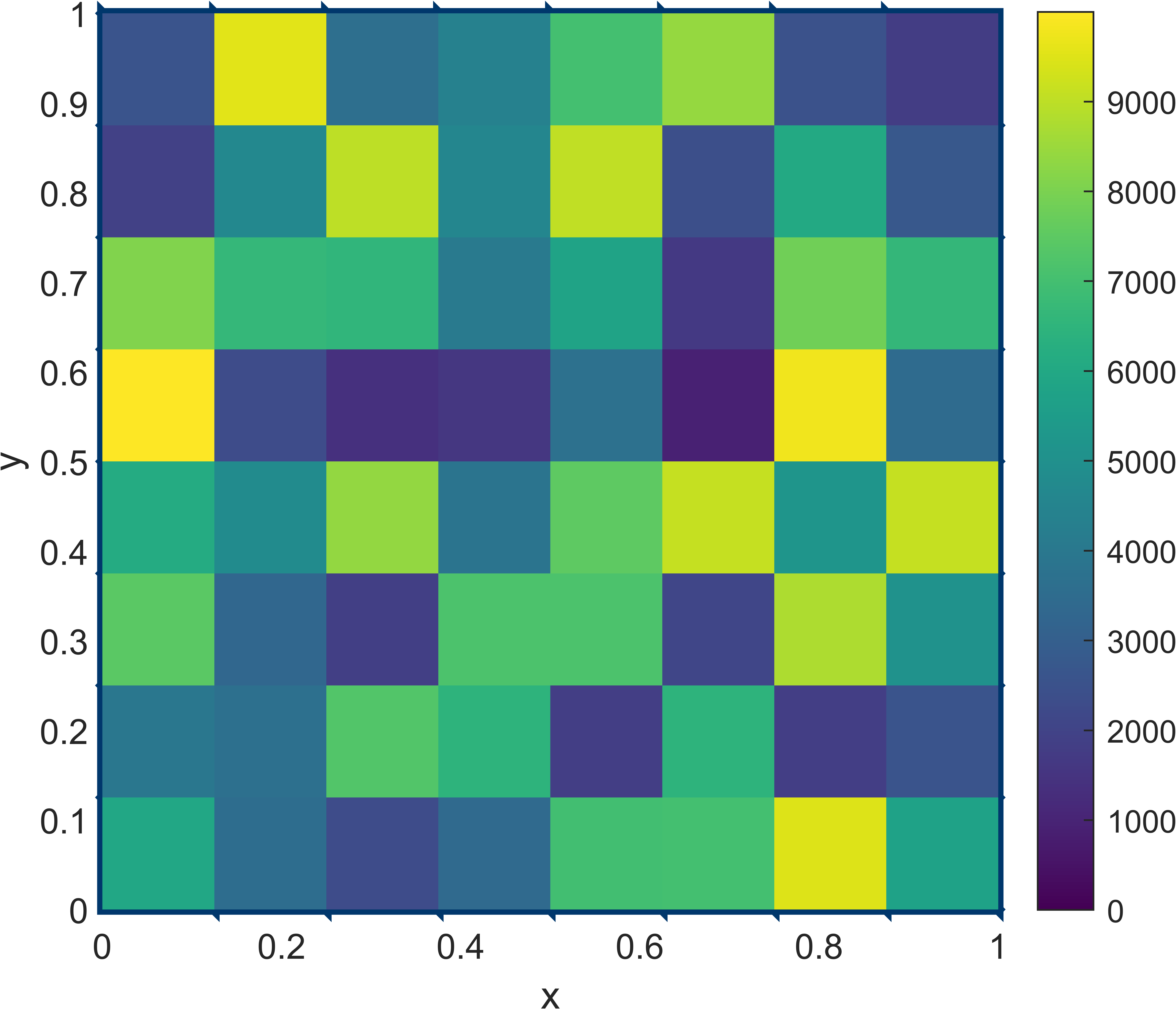}}
    \end{subfigure}

    \begin{subfigure}[t]{0.32\textwidth}
        \centering
        \scalebox{0.38}{\includegraphics{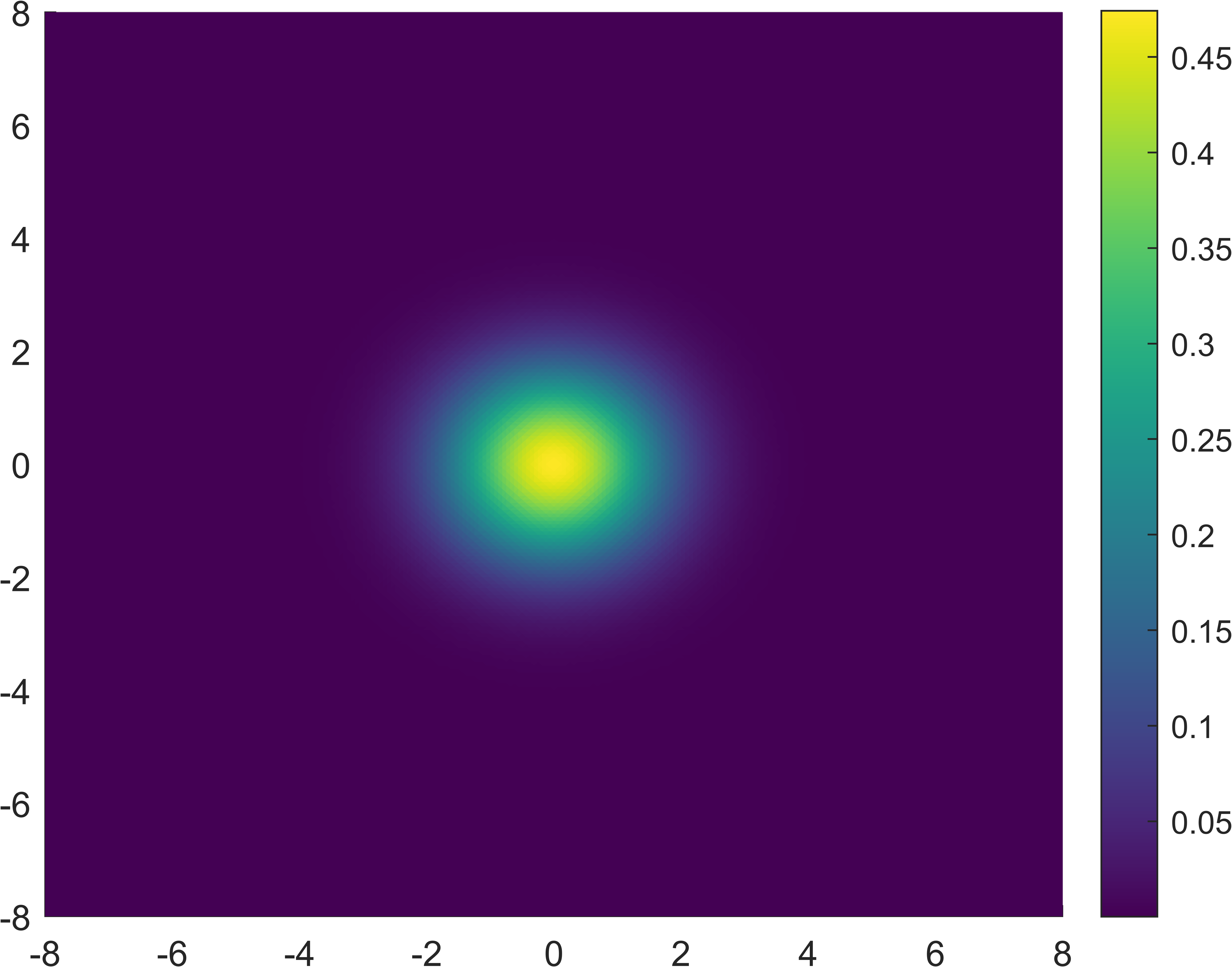}}
    \end{subfigure}
    \begin{subfigure}[t]{0.32\textwidth}
        \centering
        \scalebox{0.38}{\includegraphics{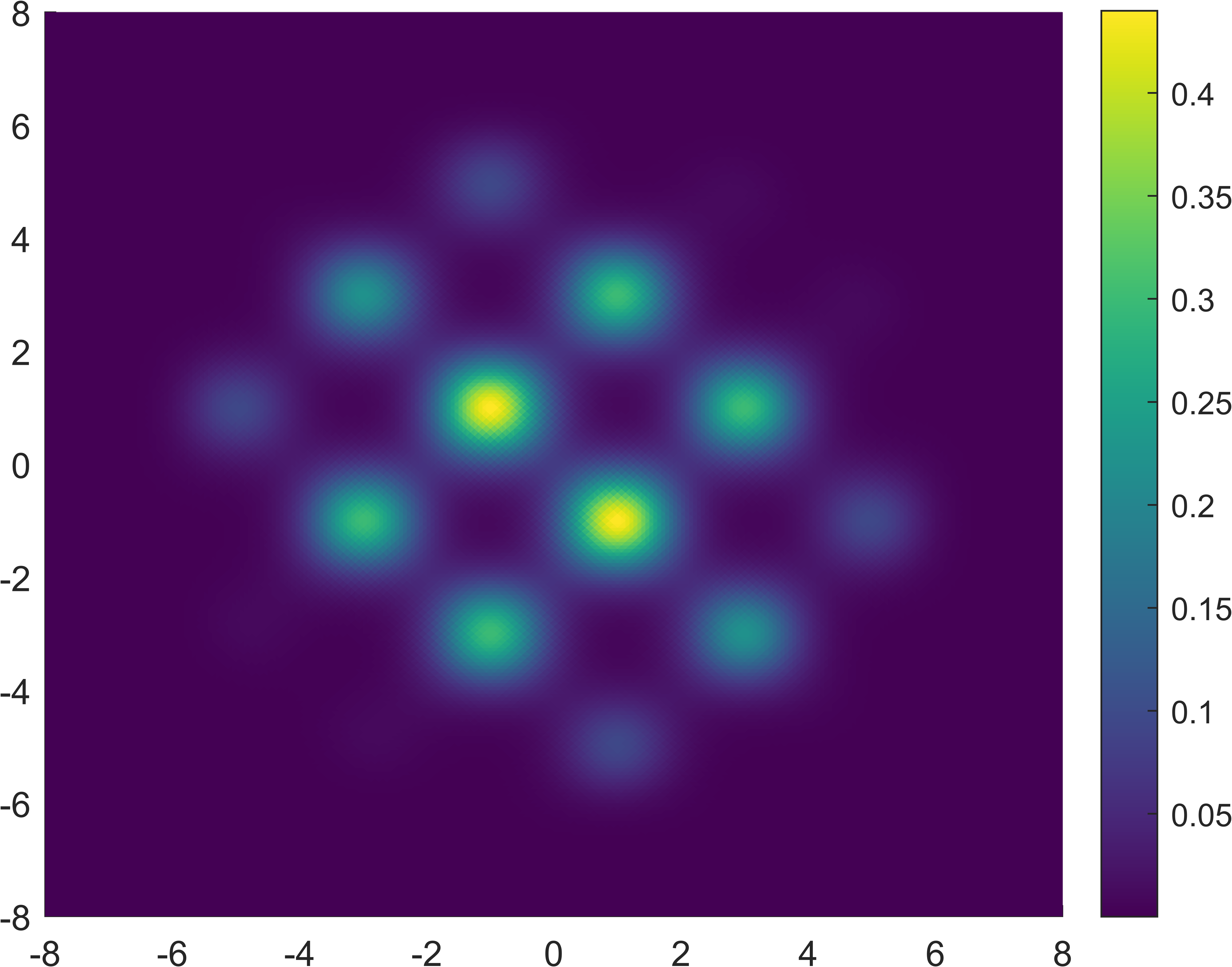}}
    \end{subfigure}
    \begin{subfigure}[t]{0.32\textwidth}
        \centering
        \scalebox{0.38}{\includegraphics{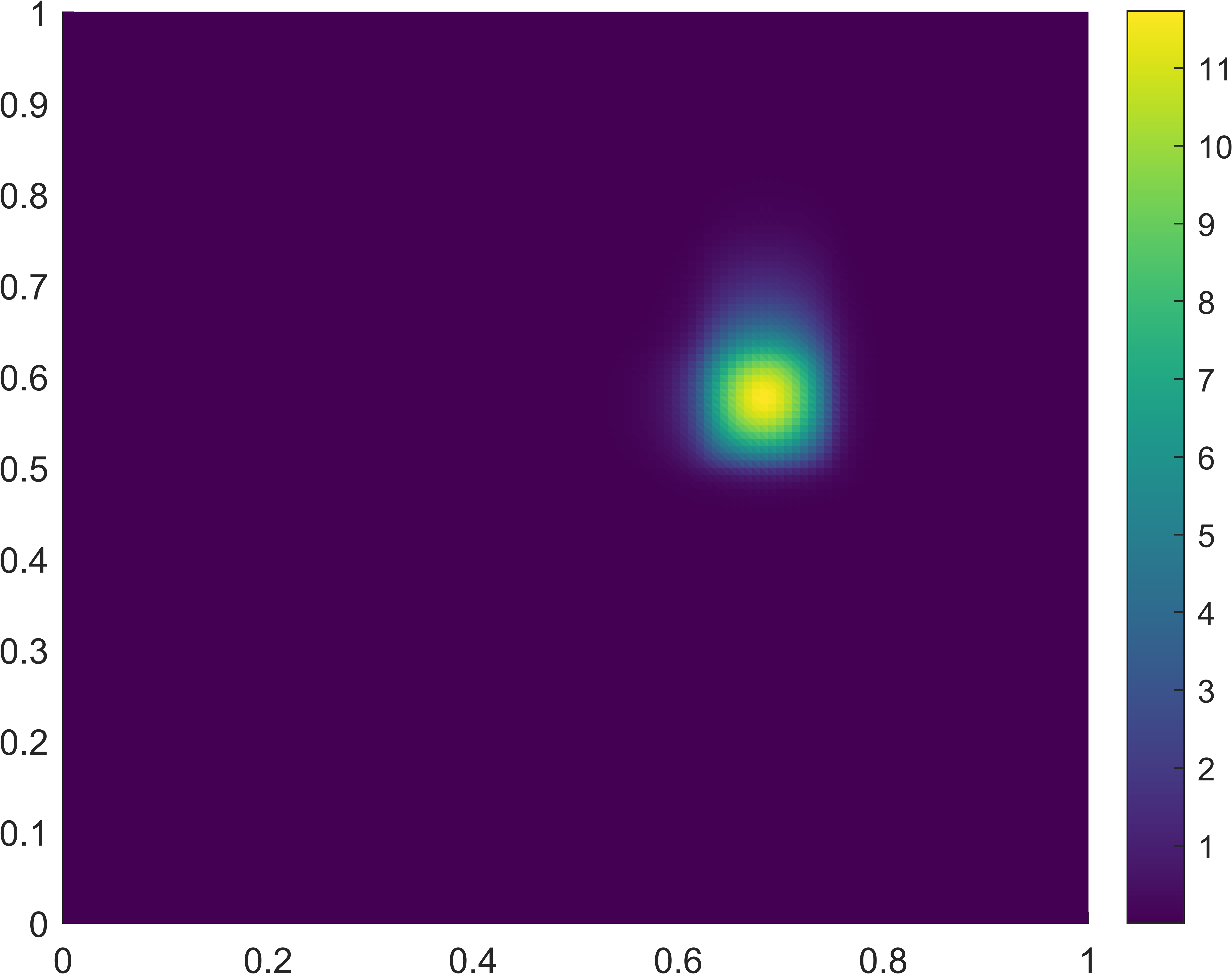}}
    \end{subfigure}
    
    \vspace{-0.1cm} 
    \caption{Projection onto piecewise constants of \(V_1,V_2,V_3\) (first row from left to right) and
    CR approximations of the corresponding ground states (second row).}
    \label{fig:groundStates}
\end{figure}
\FloatBarrier

%%% 6.4 Anderson localisation %%%
\subsection{Anderson localisation} \label{sec:7.4}
The highly disordered potential \(V_3\) first decomposes the unit square \(\Omega = (
0,1)^2\) into \(8\times8\) subsquares and second assigns to each of them a (uniformly
distributed) random value between \(0\) and \(10^4\) (by the MATLAB call \(\texttt{%
randi(}10^4\texttt{)}-1\)) displayed in Figure~\ref{fig:groundStates}. Uniform mesh-%
refinement leads to the optimal convergence rate one for GLB and GUB towards \(\lambda
_1=1.647289\times 10^3\) in Figure~\ref{fig:anderson}. The highly disordered potential
allows for exponential localisation of the ground state \cite{arnold,AP19} visible in
Figure~\ref{fig:groundStates}. A good approximation of the principal eigenpair \((
\lambda_1,u_1)\) requires a fine mesh for the support of \(u_1\). The latter is only
achieved for very fine uniform meshes as seen in Figure~\ref{fig:anderson} with optimal
convergence rate one for GLB and GUB but poor accuracy. Adaptive mesh-refinement significantly
improves the accuracy of the GUB and also \(\GLB_\es\) in Figure~\ref{fig:anderson}:
The accuracy of any adaptively computed guaranteed upper bound is improved by one order
of magnitude and so is \(\GLB_\es\). The post-processed GLB, however, remain almost
constantly close to zero in Figure~\ref{fig:anderson}. We value any post-processed
GLB under adaptive refinement as useless in this example with \(h_{\max} = 0.125\)
and \(\|V_3\|_\infty = 9996\). The adaptive mesh in Figure~\ref{fig:anderson_mesh}
depicts the local refinement in a small neighbourhood of the ground state's support
from Figure~\ref{fig:groundStates} in strong agreement with the improved accuracy of
\(\GLB_\es\) in Figure~\ref{fig:anderson}.

\vspace{-0.1cm} 
\begin{figure}[ht]
	\centering
	\scalebox{0.72}{\includegraphics{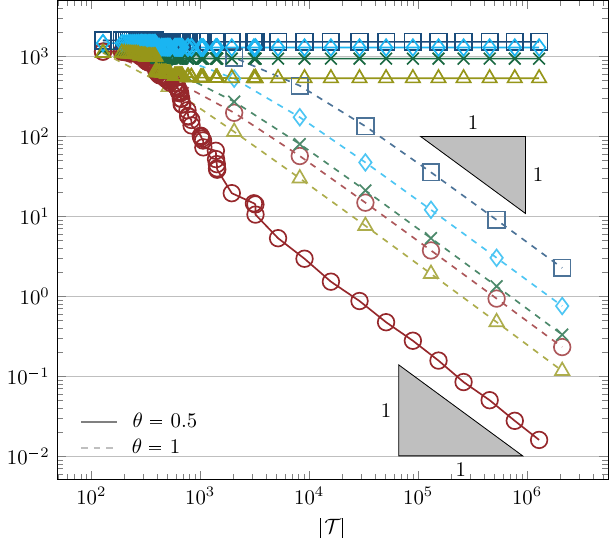}}
	\scalebox{0.72}{\includegraphics{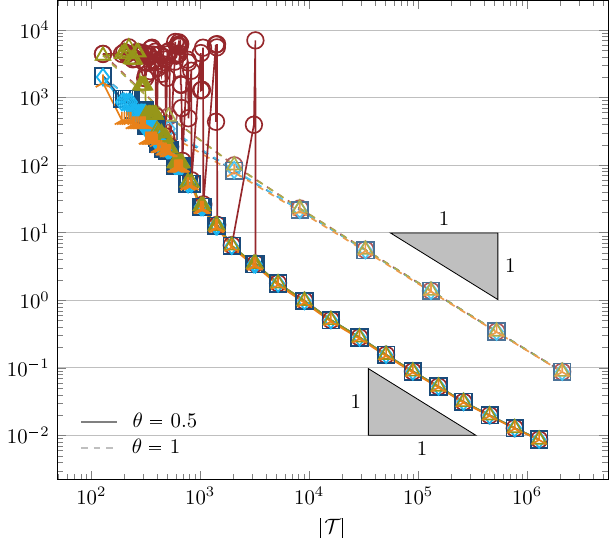}}
	\vspace{-0.3cm} 
	\caption{Convergence history plot for \(\lambda_1-\GLB(1)\) (left) and \(\GUB(1)-\lambda_1\) (right)
		on uniform (dashed) and adaptive (solid) triangulations of the unit square under the disordered
		potential \(V_3\) in Subsection \ref{sec:7.4}.}
	\label{fig:anderson}
\end{figure}

\vspace*{-0.5cm}
\begin{figure}[ht]
	\centering
	\scalebox{0.3}{\includegraphics{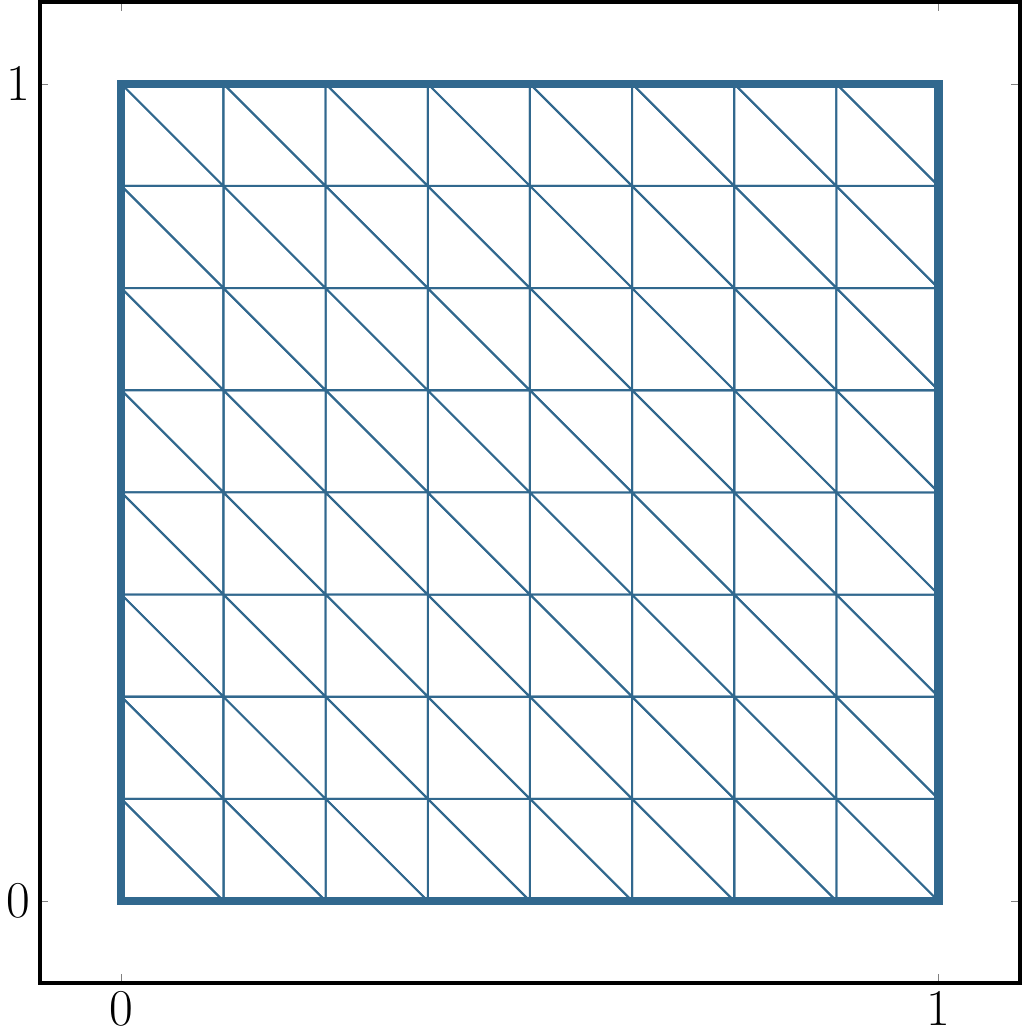}}
	\hspace{1cm}
	\scalebox{0.3}{\includegraphics{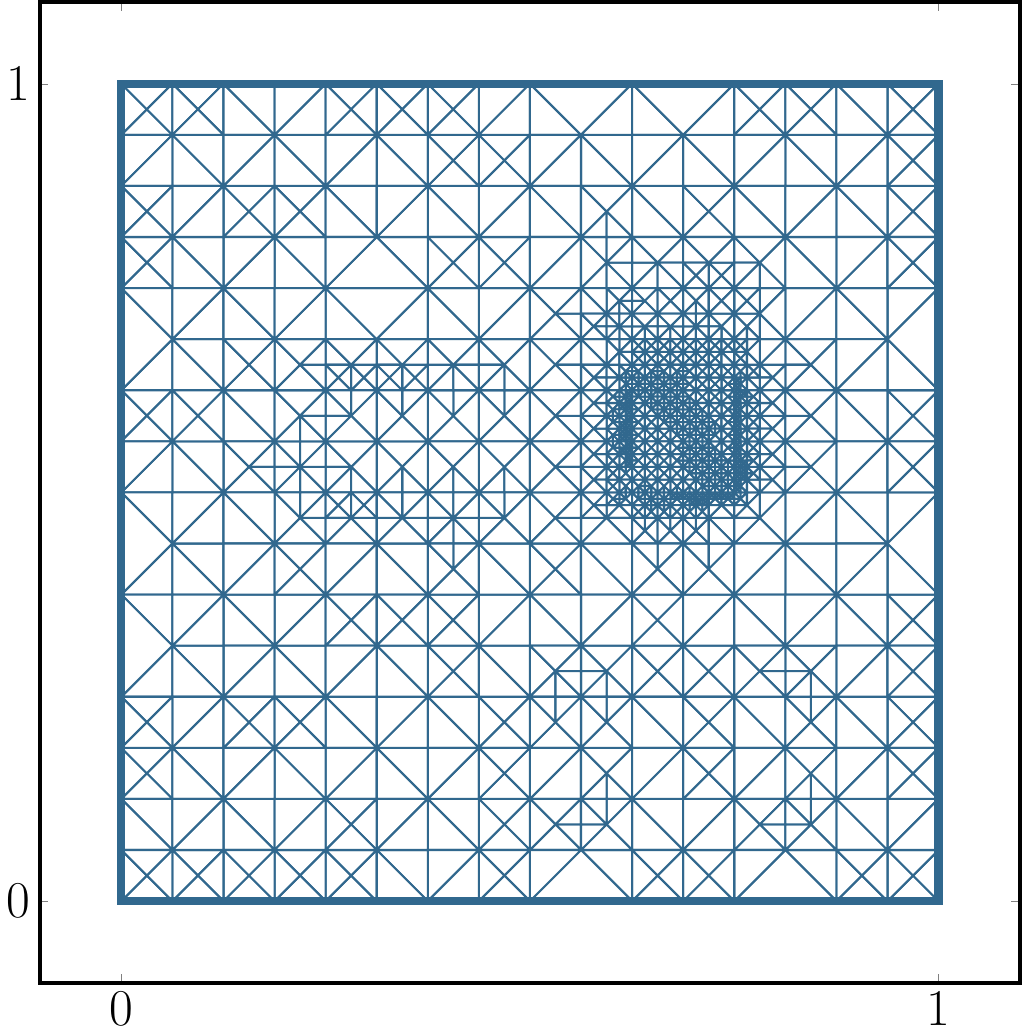}}
	\vspace{-0.2cm} 
	\caption{Initial (left) and adaptive (right) triangulations of the unit square \(\Omega = (0,1)^2\)
		in Subsection \ref{sec:7.4} with \(|\Tcal_0| = 128\) and \(|\Tcal_{70}| = 1939\).}
	\label{fig:anderson_mesh}
\end{figure}

%%% 6.5 Conclusions %%%
\subsection{Conclusions} \label{sec:7.5}
All numerical experiments provide guaranteed upper and lower eigenvalue bounds as
asserted in the theorems of Section~\ref{sec:3} and \ref{sec:6}. The novel extra-%
stabilised scheme is compatible with adaptive mesh-refinement and then superior
with (empirical) optimal convergence rates for the guaranteed bounds. It is surprising
that even the GUB converge with optimal rates on adaptive meshes in all numerical
examples. Except for only a few coarse meshes, the suggested post-processed GUB
of Section~\ref{sec:6} are of the same accuracy as the direct GUB from the Courant
EVP, that come at higher computational cost. The post-processed GLB of Section~\ref{sec:3}
are less efficient on uniform meshes than \(\GLB_\es\) and only \(\GLB_\mCR\) is
more competitive on uniform meshes. In adaptive mesh-refining the post-processed 
GLB are practically useless in the case of localisation of eigenfunctions even on
convex domains. Based on the numerical experiments in Section \ref{sec:7},
we recommend the novel adaptive sCR EVP for this class of Schr\"odinger EVP.

\paragraph{Acknowledgements.} The second author was supported from the Studienstiftung
des Deutschen Volkes and the Deutsche Forschungsgemeinschaft (DFG) under Germany's
Excellence Strategy -- The Berlin Mathematics Research Center MATH+ (EXC-2046/1,
project ID: 390685689).

\bibliography{bibliography}
\end{document}